\documentclass[11pt,reqno,a4paper]{amsart}

\usepackage{etoolbox}
\usepackage{amsmath}
\usepackage{enumerate}
\usepackage{amssymb}
\usepackage{amscd}
\usepackage{amsthm}
\usepackage{amsfonts}
\usepackage{graphicx}
\usepackage[matrix, arrow, curve]{xy}
\usepackage{comment}
\usepackage{xcolor}
\usepackage[scaled]{helvet} 
\usepackage{courier} 
\usepackage[T1]{fontenc}                                                    
\usepackage{hyperref}

\normalfont

\numberwithin{equation}{section}

\newcommand{\bR}{\mathbb{R}}

\newcommand{\bZ}{\mathbb{Z}}

\newcommand{\R}{\mathbb{R}}
\newcommand{\C}{\mathbb{C}}

\newcommand{\ra}{\rightarrow}

\newcommand{\qand}{\quad \textrm{and} \quad}

\newcommand\subsetsim{\mathrel{%
\ooalign{\raise0.2ex\hbox{$\subset$}\cr\hidewidth\raise-0.8ex\hbox{\scalebox{0.9}{$\sim$}}\hidewidth\cr}}}

\newcommand{\Z}{\mathbb Z}

\theoremstyle{theorem}
\newtheorem{theorem}{Theorem}[section]
\newtheorem{corollary}[theorem]{Corollary}
\newtheorem{proposition}[theorem]{Proposition}
\newtheorem{lemma}[theorem]{Lemma}

\newtheorem{problem}[theorem]{Problem}

\theoremstyle{definition}
\newtheorem{definition}[theorem]{Definition}

\newtheorem{remark}[theorem]{Remark}
\newtheorem{example}[theorem]{Example}

\patchcmd{\subsection}{-.5em}{.5em}{}{}
\patchcmd{\subsubsection}{-.5em}{.5em}{}{}

\begin{document}

\title{Analytic properties of approximate lattices}

\author{Michael Bj\"orklund}
\address{Department of Mathematics, Chalmers, Gothenburg, Sweden}
\email{micbjo@chalmers.se}

\author{Tobias Hartnick}
\address{Mathematisches Institut, Arndtstr. 2, 35392 Giessen, Germany}
\email{tobias.hartnick@math.uni-giessen.de}
\thanks{}



\date{\today}

\dedicatory{}

\maketitle

\begin{abstract} We introduce a notion of cocycle-induction for strong uniform approximate lattices in locally compact second countable groups and use it to relate (relative) Kazhdan- and Haagerup-type of approximate lattices to the corresponding properties of the ambient locally compact groups. Our approach applies to large classes of uniform approximate lattices (though not all of them) and is flexible enough to cover the $L^p$-versions of Property (FH) and a-(FH)-menability as well as quasified versions thereof a la Burger--Monod and Ozawa.
\end{abstract}

\section{Introduction}

\subsection{Approximate lattices} This article is concerned with analytic properties of approximate lattices in locally compact second countable (lcsc) groups, in particular with properties of Kazhdan and Haagerup type. Following Tao \cite{Tao08}, we say that a subset $\Lambda$ of a group is an \emph{approximate subgroup} if it is symmetric (i.e., $\Lambda^{-1} = \Lambda$) and contains the identity, and if moreover there exists a finite set $F_\Lambda \subset \Lambda^3$ such that
\[
\Lambda^2 \subset \Lambda F_\Lambda.
\]
While the original interest was mostly in families of \emph{finite} approximate subgroups with $F_\Lambda$ of some uniformly bounded size, here we are interested in \emph{infinite} approximate subgroups of lcsc groups. 

If $\Lambda$ is an approximate subgroup of a lcsc group $G$, then we refer to the group $\Lambda^\infty$ generated by $\Lambda$ in $G$ as the \emph{enveloping group} of $\Lambda$ and to the pair $(\Lambda, \Lambda^\infty)$ as an \emph{approximate group}. Following \cite{BH}, we say that an approximate subgroup $\Lambda \subset G$ is a \emph{uniform approximate lattice} in $\Lambda$ if it is a Delone\footnote{See the body of the text for detailed definitions.} subset of $G$. This terminology is motivated by the observation that if $\Lambda = \Lambda^\infty$ is actually a subgroup of $G$, then it is a uniform approximate lattice if and only if it is a uniform lattice.

\begin{remark} In \cite{BH} we also discussed several tentative definitions of ``non-uniform approximate lattices''. In the present article we focus exclusively on uniform approximate lattices to avoid certain integrability issues.
\end{remark}

We observed in \cite{BH}, that if $\Lambda$ is a uniform approximate lattice in $G$, then many properties of the group $G$ are reflected by properties of the approximate group $(\Lambda, \Lambda^\infty)$. For example, $G$ is compactly generated if and only $\Lambda^\infty$ is finitely generated. In this case, the canonical quasi-isometry class of $G$ is represented by the restriction of any word metric on $\Lambda^\infty$ to $\Lambda$, and $G$ is amenable if and only if $\Lambda$ is metrically amenable with respect to any such word metric.

Here we are interested in the question, in how far analytic properties of a lcsc group $G$ are reflected by analytic properties of a uniform approximate lattice $\Lambda$ in $G$ and the associated approximate group $(\Lambda, \Lambda^\infty)$. We will focus our investigation on variants of two important analytic properties of lcsc groups, namely Kazhdan's Property (T) (also known as Property (FH), see \cite{Kaz67, BHVbook}) and the Haagerup property (also known as Gromov's a-T-menability, see \cite{Ha, Gr, CCJVbook}).

\subsection{Properties of Kazhdan- and Haagerup-type}

Recall that a lcsc group $G$ has the Haagerup Property if there exists a metrically proper affine isometric action on a Hilbert space. As for Property (T), there are many equivalent characterizations. The one we are going to generalize here is known as Property (FH), which demands that every affine isometric action of $G$ on a Hilbert space has bounded orbits.
Both the Haagerup Property and Property (FH) have been generalized in two different directions.

Firstly, if instead of Hilbert spaces we consider $L^p$-spaces for some fixed $p \in (1, \infty)$, then we obtain the notion of a-F$L^p$-menability \cite{CTV}, respectively the notion of Property (F${L^p}$) as introduced and studied by Bader, Furman, Gelander and Monod \cite{BFGM}. (We warn the reader, that while Property (FH) is equivalent to Property (T), Property (F$L^p$) as just defined is strictly stronger than Property (T$_{L^p}$) from \cite{BFGM}.)

Secondly, if one replaces affine actions and cocycles by the weaker notions of quasi-cocycles, respectively weak quasi-cocycles, then one obtains strengthenings of Property (FH), which were introduced respectively by Burger--Monod \cite{BM1, Mobook} and Ozawa \cite{Oz11} under the names Property (TT) and Property (TTT).  Since these properties actually generalize Property (FH) and to avoid a conflict of terminology in the $L^p$-case, we follow the terminology introduced by Mimura \cite{Mi11} and refer to them as Property (FFH) and (FFFH) respectively. Dually, one obtains weakening of the Haagerup Property, which we refer to respectively as a-(FFH)-menability and a-(FFFH)-menability (which is more consistent with our other terminology than the term ``Ha-Ha-Haagerup property'' suggested by Nicolas Monod). Note that one can combine both generalizations in the obvious way. 
\begin{definition}\label{TypeProperties} For the purposes of this introduction we refer to the properties (F${L^p}$), (FF${L^p}$) and (FFF${L^p}$) for $p \in (1, \infty)$ as \emph{Kazhdan type properties}, and to the a-F$L^p$-amenability, a-FF$L^p$-menability and the a-FFF$L^p$-menability for $p \in (1, \infty)$ as \emph{Haagerup type properties}. (In the body of the text we will actually consider Kazhdan and Haagerup type properties with respect to more general classes of uniformly convex reflexive Banach spaces than just $L^p$-spaces; all the results presented below remain true in this more general setting.)
\end{definition}

Cornulier in his thesis \cite {Co2, Co1} introduced the notion of relative Property (T) and relative Haagerup-Property of a group with respect to a \emph{subset}. (Relative Property (T) with respect to \emph{subgroups} was defined much earlier by Margulis \cite{Mar82} and implicitly already appears in Kazhdan's original paper \cite{Kaz67}.) Namely, a lcsc group $G$ has Property (T) relative to a subset $A$ if every affine isometric action of $G$ on a Hilbert space has a bounded $A$ orbit, and it has the Haagerup property relative to a subset $A$ if there exists an affine isometric action of $G$ on a Hilbert space which is  $A$-proper. If $\mathcal P$ is any Kazhdan-type or Haagerup-type property, then relative versions can be defined similarly.
\begin{definition}\label{ApproxProperties} Let $(\Lambda, \Lambda^\infty)$ be an approximate group and let $\mathcal P$ be a Kazhdan-type of Haagerup-type property. We say that $(\Lambda, \Lambda^\infty)$ has $\mathcal P$ if $\Lambda^\infty$ has $\mathcal P$ relative to $\Lambda$.
\end{definition}
\begin{remark} Note that if the group $\Lambda^\infty$ enjoys a Kazhdan-type or Haagerup-type property $\mathcal P$, then so does the approximate group $(\Lambda, \Lambda^\infty)$. However, as we will see in Example \ref{ExamplesIntro} below, in general the converse is not true.
\end{remark}
In terms of Definition \ref{TypeProperties} and Definition \ref{ApproxProperties}, the main problem considered in the present article can be formulated as follows.
\begin{problem} Let $G$ be a lcsc group and $\Lambda \subset G$ a uniform approximate lattice with enveloping group $\Lambda^\infty$. How are the Kazhdan-type and Haagerup-type properties of $G$ related to the Kazhdan-type and Haagerup-type properties of $(\Lambda, \Lambda^\infty)$?
\end{problem}

\subsection{Main results}
Recall that if $\Gamma$ is a uniform lattice in a lcsc group $G$ and $\mathcal P$ is a property of Kazhdan-type or Haagerup-type, then $G$ has $\mathcal P$ if and only if $\Gamma$ has $\mathcal P$. There are a number of problems which prevent us from generalizing the results to arbitrary uniform approximate lattices.

Firstly, note that if $\Gamma$ is a uniform lattice in $G$, then the homogeneous space $G/\Gamma$ always admits a unique $G$-invariant probability measure. If $\Lambda$ is merely a uniform approximate lattice in $G$, then it is currently not known in full generality whether the natural substitute for the homogeneous space $G/\Gamma$, the so-called hull $X_\Lambda$ of $\Lambda$, admits a $G$-invariant probability measure. Let us call $\Lambda$ a \emph{strong uniform approximate lattice} if such a measure exists. The following theorem will be established in Subsection \ref{SecHaagerup}:
\begin{theorem}\label{ThmHaagerupIntro} Let $\mathcal H$ be a Haagerup-type property, $G$ be a lcsc group and $\Sigma \subset G$ be a uniform approximate lattice which is contained in a strong uniform approximate lattice $\Lambda \subset G$. Then $G$ has $\mathcal H$ if and only if $(\Sigma, \Sigma^\infty)$ has $\mathcal H$.
\end{theorem}
We emphasize that we do not need to assume that $\Sigma$ itself is strong, but only that it is contained in a strong uniform approximate lattice. As far as examples are concerned, let us briefly recall the main constructions of approximate lattices. For this let $G$, $H$ be lcsc groups and denote by $\pi_G: G\times H \to G$ the projection onto the first factor. Moreover, let $\Gamma < G \times H$ be a uniform lattice which projects injectively into $G$ and densely into $H$, and let $W$ be a compact identity neighbourhood in $H$ with dense interior. Then the set
\[
\Lambda := \Lambda(G, H, \Gamma, W) := \pi_G(\Gamma \cap (G \times W))
\]
is called a \emph{symmetric model set}. A symmetric model set is called \emph{regular} provided $W$ is Jordan-measurable with dense interior, aperiodic (i.e., ${\rm Stab}_H(W) = \{e\}$) and satisfies $\partial W \cap \pi_H(\Gamma) = \emptyset$. A subset $\Delta$ of a lcsc group $G$ is called a \emph{Meyer set} if there exists a symmetric model set $\Lambda \subset G$ and a finite subset $F \subset G$ such that $\Delta \subset \Lambda \subset \Delta F$. In fact, $\Lambda$ can be chosen to be regular.
\begin{proposition}[\cite{BH}]\label{MeyerApproxGrp}
 Every regular symmetric model set is a strong uniform approximate lattice, and every symmetric Meyer set is a uniform approximate lattice (not necessarily strong).
\end{proposition}
Combining Proposition \ref{MeyerApproxGrp} and Theorem \ref{ThmHaagerupIntro} we obtain:
\begin{corollary} Let $\mathcal H$ be a Haagerup-type property, $G$ be a lcsc group and $\Lambda \subset G$ be a symmetric Meyer set containing the identity. Then $G$ has $\mathcal H$ if and only if $(\Lambda, \Lambda^\infty)$ has $\mathcal H$.
\end{corollary}
If $G$ is abelian, then every uniform approximate lattice in $G$ is a symmetric Meyer set \cite{Meyerbook, Moody}. As for general $G$, it is currently not known whether there exist any uniform approximate lattices which are not symmetric Meyer sets. In particular, the corollary covers all currently known examples of uniform approximate lattices.

The situation with Kazhdan-type properties is more complicated than in the Haagerup case. To formulate our result, we say that a model set $\Lambda$ is of \emph{almost connected type}, respectively \emph{connected Lie type}, if $\Lambda = \Lambda(G, H, \Gamma, W)$ for some $H$ which is almost connected (i.e., connected-by-compact), respectively a connected Lie group. The following theorem will be established in Subsection \ref{SecTMeyer}:
\begin{theorem}\label{KazhdanType} Let $\mathcal T$ be a Kazhdan-type property, $G$ be a lcsc group and $\Lambda \subset G$ be a uniform approximate lattice. Assume that one of the following holds:
\begin{enumerate}
\item $\Lambda$ is a model set.
\item $\Lambda$ is a Meyer set which is contained in a model set of almost connected type.
\end{enumerate}
Then $G$ has $\mathcal T$ if and only if $(\Lambda, \Lambda^\infty)$ has $\mathcal T$.
\end{theorem}
Concerning Case (2) of Theorem \ref{KazhdanType}, it is natural to ask how restrictive the assumption on a Meyer set to be contained in a model set of almost connected type actually is. While it is not true that every Meyer set is contained in a model set of almost connected type, we establish (a more precise version of) the following result in the appendix:
\begin{theorem} Let $\Lambda$ be an arbitrary Meyer set. Then $\Lambda$ is contained in a finite union of left-translates of a model set of almost connected type. In fact, it is even contained in a finite union of left-translates of a model set of connected Lie type.
\end{theorem}
\begin{remark} For Property (T), Part (1) of Theorem \ref{KazhdanType} was established (without using the language of approximate groups) by Chifan and Ioana \cite[Cor.\ 1.3]{CI} using Cornulier's notion of resolutions \cite{Co1}. While their proof is simpler than ours, it uses unitary representation theory, and it is unclear to us whether their approach can be extended to cover Properties (FFH) and (FFFH) and/or $L^p$-Kazhdan-type properties for $p \neq 2$.
\end{remark}
Theorem \ref{ThmHaagerupIntro} and Theorem \ref{KazhdanType} provide a rich supply of examples of approximate groups with various Haagerup-type and Kazhdan-type properties, see Examples \ref{ExTnoT} and \ref{ExSp} below.
\begin{example}\label{ExamplesIntro} 
\begin{enumerate}[(i)]
\item Let $n \geq 2$, $G := {\rm SU}(n,2)$ and $H := {\rm SU}(n+1,1)$, let $\Gamma$ be an irreducible lattice in $G \times H$ and let $\Lambda = \Lambda(G, H, \Gamma, W)$ be a regular model set in $G$ with symmetric  window $W \subset H$ containing the identity. It then follows from the results in \cite{BM1} that $\Lambda$ has the Burger--Monod Property (FF$H$) (a.k.a. Property (TT)), whereas $\Lambda^\infty$ does not even have Property (FH). Moreover, by results from \cite{BFGM}, $\Lambda$ also has Property (F$L^p$) for all $p \in (1, \infty)$, whereas $\Lambda^\infty$ does not have Property (F$L^p$) for any $p \in (1, \infty)$. (For $p = 2$ this follows also from the results in \cite{CI}.)
\item Let $n \geq 2$ and let $\Lambda$ be a symmetric model set in ${\rm Sp}(n,1)$ containing the identity. Then $\Lambda$ has Property (FH), but by results from \cite{Oz11}, respectively \cite{CTV}, it is both a-FFF$H$-menable and a-F$L^p$-menable for $p>4n+2$.
\end{enumerate}
\end{example}

\subsection{Towards applications} Given approximate groups $(\Lambda,\Lambda^\infty)$ and $(\Sigma, \Sigma^\infty)$ we say that a map of pairs $\varphi: (\Lambda, \Lambda^\infty) \to (\Sigma, \Sigma^\infty)$, is called a \emph{morphism} if $\varphi(gh) = \varphi(g)\varphi(h)$ for all $g, h \in \Lambda^\infty$ and a \emph{quasimorphism} if the set $\{\varphi(h)^{-1}\varphi(g)^{-1}\varphi(gh) \mid g, h\in \Lambda^\infty\}$ is finite. Then the following result can be established as in the group case. 
\begin{proposition}\label{Rigidity} Let $p \in (1, \infty)$, $(\Lambda,\Lambda^\infty)$, $(\Sigma, \Sigma^\infty)$ be approximate groups and $\varphi: (\Lambda, \Lambda^\infty) \to (\Sigma, \Sigma^\infty)$. Assume that either of the following holds:
\begin{enumerate} [(i)]
\item If $(\Lambda, \Lambda^\infty)$ has (F$L^p$), $(\Sigma, \Sigma^\infty)$ is a-F$L^p$-menable and $\varphi$ is a morphism.
\item  If $(\Lambda, \Lambda^\infty)$ has (FFF$L^p$), $(\Sigma, \Sigma^\infty)$ is a-FFF$L^p$-menable and $\varphi$ is a quasimorphism.
\end{enumerate}
Then $\varphi(\Lambda)$ is finite.
\end{proposition}
\begin{remark} In view of the Proposition it would be of interest to have more examples of approximate groups which have (FFF$L^p$) or are a-FFF$L^p$-menable. For instance, it would be of interest to know, whether every higher rank Lie group has Property (FFF$H$). Similarly, one would like to know whether every (coarsely-connected, finitely generated) hyperbolic approximate group $(\Lambda, \Lambda^\infty)$ is a-FFF$H$-menable by an argument similar to the one suggested in \cite{Oz11} for groups.
\end{remark}
Finally, let us suggest our own variations on Property (T) and the Haagerup property. A natural class of maps between approximate groups is given by $2$-Freiman quasimorphisms, i.e., maps $\varphi: (\Lambda, \Lambda^\infty) \to (\Sigma, \Sigma^\infty)$ such that the set $\{\varphi(h)^{-1}\varphi(g)^{-1}\varphi(gh) \mid g, h\in \Lambda\}$ is finite. To obtain a result analogous to Proposition \ref{Rigidity} for such maps, we suggest to introduce the following extensions of Property (TT) and (TTT):
\begin{definition} Let $(\Lambda, \Lambda^\infty)$ be an approximate group, $E$ be an $L^p$-space, and $\pi: \Lambda^\infty \to O(E)$ be an arbitrary map. Then a map $b: \Lambda^\infty \to E$ is called a \emph{$2$-Freiman wq-cocycle} if
\[
\sup_{g,h \in \Lambda} \|b(gh) - b(g) - \pi(g)b(h)\| < \infty.
\]
We say that $(\Lambda, \Lambda^\infty)$ has Property (FFFF$L^p$) if every $2$-Freiman wq-cocycle on $\Lambda^\infty$ is bounded on $\Lambda$, and that it is a-FFFF$L^p$-menable if there exists a  $2$-Freiman wq-cocycle on $\Lambda^\infty$ which is proper on $\Lambda$.
\end{definition}
We leave the following as an exercise:
\begin{proposition} Let $p \in (1, \infty)$, $(\Lambda,\Lambda^\infty)$, $(\Sigma, \Sigma^\infty)$ be approximate groups and $\varphi: (\Lambda, \Lambda^\infty) \to (\Sigma, \Sigma^\infty)$. Assume that $(\Lambda, \Lambda^\infty)$ has (FFFF$L^p$), $(\Sigma, \Sigma^\infty)$ is a-FFFF$L^p$-menable and $\varphi$ is a 2-Freiman quasimorphism.
Then $\varphi(\Lambda)$ is finite.
\end{proposition}
In relation to this, we would like to advertise the following problem:
\begin{problem} Find examples of approximate groups with Property  (FFFF$L^p$). Also consider the case of higher order $k$-Freiman cocycles instead of $2$-Freiman wq-cocycles.
\end{problem}
\subsection{On the method of proof}
The proofs of Theorem \ref{ThmHaagerupIntro} and Theorem \ref{KazhdanType} are based on a version of cocycle induction from a strong uniform approximate lattice $\Lambda$ to the ambient lcsc group $G$, which may be of independent interest. This construction is general enough to also apply to (weak) quasi-cocycles with values in $L^p$-spaces. Theorem \ref{ThmHaagerupIntro} then follows from the observation that induction preserves properness of cocycles. On the other hand, it is not obvious (even in the Hilbert setting) that boundedness of the induced cocycle implies boundedness of the original cocycle on $\Lambda$. In case where $\Lambda = \Lambda(G, H, \Gamma, W)$ is a model set, we can circumvent this problem by using an alternative model for induction, which allows us to extend the induced $G$-cocycle to a cocycle on $G \times H$. Using results from \cite{BHP1} we are then able to transfer the problem to a problem about the homogeneous space $(G\times H)/\Gamma$. Even after the reduction to the homogeneous setting, the proof of Theorem \ref{KazhdanType} remains non-trivial and relies on a strengthening of arguments developed by Ozawa in his work on Property (TTT).

%

\tableofcontents

\subsection*{Acknowledgements}
We thank Marc Burger for pointing out the potential relevance of \cite{Oz11} for approximate groups and Yves de Cornulier for explanations concerning his work on relative Property (T). We thank the Departments of Mathematics at Technion and Chalmers, and the Gothenburg Center of Advanced Studies for their hospitality during our mutual visits. We thank the anonymous referees for detailed comments and suggestions.

\section{Approximate groups and approximate lattices}
\label{approxgrps}

\subsection{Basic definitions}

\begin{definition}
An \emph{approximate group} is a pair $(\Lambda, \Lambda^\infty)$, where $ \Lambda^\infty$ is a group and $\Lambda \subset \Lambda^\infty$ is a subset which satisfies the following conditions.
\begin{enumerate}[({AG}1)]\label{DefAppGrp}\label{DefApproxSubgp}
\item $\Lambda$ is symmetric, i.e., $\Lambda = \Lambda^{-1}$, and contains the identity.
\item There exists a finite subset $F_\Lambda \subset  \Lambda^\infty$ such that $\Lambda^2 \subset \Lambda F_\Lambda$. 
\item $\Lambda$ generates $\Lambda^\infty$.
\end{enumerate}
\end{definition}
If $(\Lambda, \Lambda^\infty)$ is an approximate group and $G$ is a group containing $\Lambda^\infty$, then the subset $\Lambda \subset G$ is called an \emph{approximate subgroup} of $\Gamma$. We then call $\Lambda^\infty$ the \emph{enveloping group} of $\Lambda$ in $G$.

We will be particularly interested in approximate subgroups of locally compact second countable (lcsc) groups. Recall that if $G$ is a lcsc group, then $G$ admits a metric which is 
left-invariant, proper and defines the given topology of $G$. We refer to any such metric as a \emph{left-admissible} metric on $G$. Also recall that if $(X,d)$ is an metric space, then a subset $\Lambda \subset X$ is called a \emph{Delone set} if there exist constants $R>r>0$ (called the \emph{Delone parameters} of $\Lambda$) such that
\begin{itemize}
\item  $\Lambda$ is $r$-\emph{uniformly discrete} i.e., $d(x, y) \geq r$ for all $x,y \in \Lambda$;
\item $\Lambda$ is $R$-\emph{relatively dense} in $X$, i.e., for every $x \in X$ there exists $y \in \Lambda$ with $d(x,y) \leq R$.
\end{itemize}
If $G$ is a lcsc group, then a subset $\Lambda \subset G$ is called a Delone set if it is a Delone set with respect to some left-admissible metric on $G$. One can show (see e.g. \cite[Prop. 2.2]{BH}) that this notion does not depend on the choice of left-admissible metric.
\begin{definition} An approximate subgroup $\Lambda$ of a lcsc group $G$ is called a \emph{uniform approximate lattice} provided $\Lambda$ is a Delone set in $G$.
\end{definition}
Note that a \emph{subgroup} $\Lambda < G$ is a Delone set if and only if it is a uniform lattice, hence the name.
\begin{remark}\label{Syndetic} Let $G$ be a lcsc group. Then a subset $A \subset G$ is relatively dense, if and only if there exists a compact subset $K \subset G$ such that $AK = G$ (see \cite[Prop. 2.2]{BH}). Now assume that $\Lambda \subset G$ is a Delone subset and let $\Sigma \subset \Lambda$. Then the following are equivalent (see \cite[Cor. 2.10]{BH}):
\begin{enumerate}[(i)]
\item $\Sigma$ is relatively dense in $G$.
\item There exists a finite set $F \subset \Lambda^{-1}\Lambda$ such that $\Sigma \subset \Lambda \subset \Sigma F$.
\end{enumerate}
In this situation we say that $\Sigma$ is \emph{left-syndetic} in $\Lambda$. Note that every symmetric left-syndetic subset of an approximate lattice which contains the identity is again an approximate lattice.
\end{remark}

\subsection{The hull of an approximate lattice}

For the rest of this section we fix a uniform approximate lattice $\Lambda$ in a lcsc group $G$.  We denote by $\mathcal C(G)$ the collection of all closed subsets of $G$. Then $G$ acts on $\mathcal C(G)$ by left-translations, and we denote the orbit of an element $P \in \mathcal C(G)$ by $G.P$. The set $\mathcal C(G)$ carries a compact Hausdorff topology called the \emph{Chabauty-Fell topology}, whose basic open sets are given by
\[
\{P \in \mathcal C(G) \mid P \cap K = \emptyset\} \quad \text{and} \quad \{P \in \mathcal C(G) \mid P \cap V \neq \emptyset\},
\]
where $K$ ranges over all compact subsets of $G$ and $V$ ranges over all open subsets of $G$. We refer the reader to \cite{BP, BH, BHP1} for a detailed discussion of this topology.

\begin{definition} Given an approximate lattice $\Lambda \subset G$, the closure $X_\Lambda := \overline{G.\Lambda}$ of the left-translation orbit $G.\Lambda$ in $\mathcal C(G)$ with respect to the Chabauty-Fell topology is called the \emph{hull} of $\Lambda$.
\end{definition}
Since the action of $G$ on $\mathcal C(G)$ is jointly continuous, so is the action of $G$ on the hull of $\Lambda$. We will often use the following fact (see \cite[Lemma 4.6]{BH}):
\begin{lemma}\label{P-1P} If $P \in X_\Lambda$, then $P^{-1}P \subset \Lambda^{-1}\Lambda = \Lambda^2$, hence $P$ and $P^{-1}P$ are uniformly discrete.\qed
\end{lemma}
We will also need the following fact:
\begin{lemma}\label{UniformDeloneHull} Let $\Lambda \subset G$ be Delone with parameters $R>r>0$. Then every $P \in X_\Lambda$ is Delone with parameters $R>r>0$.
\end{lemma}
\begin{proof} Let $P \in X_\Lambda$ and fix $g_n \in G$ such that $g_n\Lambda \to P$. For every $x \in G$ and $n \in \mathbb N$ there then exists $\lambda_n \in \Lambda\cap B_R(g_n^{-1}x)$. A subsequence $g_{n_k}\lambda_{n_k}$ then converges to some $p \in B_R(x)$, and then $p \in P$ by \cite[Prop. E.1.2]{BP}, showing that $P$ is $R$-relatively dense. Conversely, if $x,y \in P$ are distinct, then by the same proposition there exist $x_n, y_n \in \Lambda$ such that $g_nx_n \to x$ and $g_ny_n \to y$. For sufficiently large $n$ we then have $x_n \neq y_n$ and hence $d(x_n, y_n) \geq r$, which implies $d(x,y) \geq r$, showing that $P$ is $r$-uniformly discrete.
\end{proof}

Note that the hull $X_\Lambda$ is compact and that the $G$-action on $X_\Lambda$ is jointly continuous. This implies in particular, that $X_\Lambda$ admits a $\mu$-stationary probability measure for every admissible probability measure $\mu$ on $G$. However, if $G$ is non-amenable, then  it is unclear whether the hull supports any $G$-\emph{invariant} probability measure.
\begin{definition} $\Lambda$ is called a \emph{strong uniform approximate lattice} if the hull $X_\Lambda$ admits a $G$-invariant probability measure.
\end{definition}
\subsection{Bounded Borel sections over the hull}
We keep the notation of the previous subsection. Moreover, we fix a left-admissible metric $d$ on $G$ and denote by $R>r>0$ the Delone parameters of $\Lambda$ with respect to $d$. In analogy with the group case we define:
\begin{definition}  A map $s: X_\Lambda \to G$ is called a \emph{section} provided $s(P) \in P$ for all $P \in X_\Lambda$. A section is called \emph{bounded} if its image is pre-compact.\end{definition}
\begin{proposition}\label{ExSection} Let $\Lambda \subset G$ be a uniform approximate lattice. Then there exists a bounded Borel section $s: X_\Lambda \to G$. In fact, $s$ can be chosen to take values in $B_{2R}(e)$.
\end{proposition}
\begin{proof} It follows from Lemma \ref{UniformDeloneHull} that for every $P \in X_\Lambda$ we can choose $x \in P \cap B_R(e)$ and define an open neighbourhood of $P$ in $X_\Lambda$ by $U(P, x) := \{Q \in X_\Lambda \mid Q \cap B_{r/2}(x) \neq \emptyset\}$. Note that if $Q \in U(P,x)$, then since $Q$ is $r$-discrete, there exists a unique point $\sigma_{P, x}(Q) \in Q \cap \overline{B_{r/2}(x)}$ (which is in fact contained in $B_{r/2}(x)$) and thus we obtain a section $\sigma_{P, x}: U_{P, x} \to G$ over $U_{P,x}$. It follows from \cite[Corollary 4.7]{BH} (applied with $K := \overline{B_{r/2}(x)}$) that if $Q_n \in U(P, x)$ converge to $Q \in U(P,x)$, then  $\sigma_{P, x}(Q_n) \to \sigma_{P, x}(Q)$, hence $\sigma_{P, x}$ is continuous.

Since the open sets $U(P, x)$ cover the compact space $X_\Lambda$ there exist finitely many elements $P_1, \dots, P_n \in X_\Lambda$ and $x_i \in P_i$ such that
\[
X_\Lambda = U(P_1, x_1) \cup \dots \cup U(P_n, x_n).
\]
We thus obtain a Borel section over $X_\Lambda$ by setting
\[
\sigma(Q) := \sigma_{P_{j_Q}, x_{j_Q}}(Q), \text{ where } j_Q := \min\{j \in \{1, \dots, n\} \mid Q \in U(P_j, x_j)\}.
\]
By definition, $\min d(\sigma(Q), x_j) \leq r/2$, and since $x_j \in B_R(e)$ we deduce that $\sigma(Q) \in B_{R+r}(e) \subset B_{2R}(e)$.
\end{proof}
\begin{lemma}\label{LemmabetaCocycle} Let $s: X_\Lambda \to G$ be a Borel section. Then for all $g \in G$ and $P \in X_\Lambda$ we have
\[
\beta_s(g, P) := s(gP)^{-1}gs(P) \in \Lambda^2,
\] 
and the function $\beta_s: G \times X_\Lambda \to \Lambda^2$ satisfies the cocycle identity
\begin{equation}\label{betaCocycle}
\beta_s(g_1g_2, P) = \beta_s(g_1, g_2P) \beta(g_2, P) \quad (g_1, g_2 \in G, P \in X_\Lambda).
\end{equation}
Moreover, the image of $\beta_s$ is uniformly discrete in $G$. If $s$ is bounded, then for every compact subset $K \subset G$ the image $\beta_s(K \times X_\Lambda)$ of $K \times X_\Lambda$ is finite.
\end{lemma}
\begin{proof} It follows from Lemma \ref{P-1P} that
\[
s(gP)^{-1}gs(P) \in (gP)^{-1}gP = P^{-1}P \subset \Lambda^2.
\]
Moreover, the identity
\[
 (s(g_1g_2P)^{-1}g_1s(g_2P))(s(g_2P)^{-1}g_2s(P)) =  s(g_1g_2P)^{-1}g_1g_2s(P)
\]
implies \eqref{betaCocycle}. Finally note that if $\Lambda \subset G$ is an approximate lattice, then so is $\Lambda^2$, and thus $\beta_s$ takes values in a uniformly discrete subset of $G$. In particular, if $g$ varies over a compact set $K$ and $s$ is chosen to take values in a bounded set $B$, then $\beta_s(g, P)$ is contained in set $B K B \cap \Lambda^2$, which is finite.
\end{proof}
In the sequel we refer to $\beta_s$ as the \emph{cocycle associated with the section $s$}.  Note that the cocycle identity \eqref{betaCocycle} implies that $\beta(e, P) = e$ (since $\beta(e, P) = \beta(e, P)^2$), hence for all $g \in G$ we have $e = \beta(g^{-1}g, P) = \beta(g^{-1}, gP)\beta(g, P)$, i.e.,
\begin{equation}\label{betaInverse}
\beta(g, P) = \beta(g^{-1}, gP)^{-1}.
\end{equation}
We also record the following standard fact for later reference:
\begin{lemma}\label{ChangeOfSection} If $s_1, s_2: X_\Lambda \to G$ are Borel sections, then 
\[
\beta_{s_2}(g, P) =u(gP)^{-1}\beta_{s_1}(g,P)u(P),
\]
where $u: X_\Lambda \to \Lambda^2$ is given by $u(P) :=  s_1(P)^{-1}s_2(P)$.
\end{lemma}
\begin{proof} By Lemma \ref{P-1P} we have  $s_1(P)^{-1}s_2(P)\in P^{-1}P \subset \Lambda^{-1}\Lambda = \Lambda^2$, hence $u$ is well-defined. The formula relating $\beta_{s_1}$ and $\beta_{s_2}$ then follows from
\[ s_2(gP)^{-1}gs_2(P)\quad = \quad s_2(gP)^{-1}s_1(gP)(s_1(gP)^{-1}gs_1(P)) s_1(P)^{-1}s_2(P).\]
\end{proof}

Now let $s: X_\Lambda \to G$ be a bounded Borel section taking values in $B_{2R}(e)$. Given an element $g \in G$ we denote $\|g\| := d(g,e)$, and given $g \in G$ and $\lambda \in \Lambda^2$ we define
\begin{equation}\label{Xglambda}
X_\Lambda(g, \lambda) := \{P \in X_{\Lambda}\mid \beta(g^{-1}, P)^{-1} = \lambda\}.
\end{equation}
The following observation will be used in the definition of an induced affine isometric action.
\begin{lemma}\label{HaagerupTriangleInequality} If $X_\Lambda(g, \lambda) \neq \emptyset$, then 
\[
\|g\| - 4R \quad<\quad \|\lambda\| \quad < \quad \|g\|+4R.
\]
\end{lemma}
\begin{proof} If $P \in X_\Lambda(g, \lambda)$, then by definition
\[
\lambda = s(P)^{-1}g^{-1}s(gP),
\]
with $\{\|s(P)\|, \|s(gP)\|\} \subset [0, 2R)$. Thus a simple application of the triangle inequality yields
\begin{eqnarray*}
\|\lambda\| &=& d(e, s(P)^{-1}g^{-1}s(gP)) \quad \geq \quad  d(e, s(P)^{-1}g^{-1})-  d(s(P)^{-1}g^{-1}s(gP),  s(P)^{-1}g^{-1})\\
&=&  d(s(P), g^{-1}) - d(s(gP), e) \quad \geq \quad d(g^{-1}, e) - d(e, s(P)) - d(e, s(gP))\\
&>& \|g\| -4R,
\end{eqnarray*}
and a similary argument yields the upper bound.
\end{proof}
\subsection{Approximate lattices from model sets}
An important class of approximate lattices is given by symmetric model sets \cite{Meyerbook, BHP1, BHP2} in the sense of Meyer. These examples will play an important role in the sequel, hence we briefly recall their definition and basic properties.
\begin{definition} A \emph{cut-and-project-scheme} is a triple $(G, H, \Gamma)$ where $G$ and $H$ are lcsc groups and $\Gamma < G \times H$ is a lattice which projects injectively to $G$ and densely to $H$. A cut-and-project scheme is called \emph{uniform} if $\Gamma$ is moreover a uniform lattice.
\end{definition}
\begin{remark}\label{ModelSetConventions}
In the sequel when given a \emph{uniform} cut-and-project scheme $(G, H, \Gamma)$ we will always use the following notations and conventions: Firstly, we denote by $\pi_G$, $\pi_H$ the coordinate projections of $G \times H$ and set $\Gamma_G := \pi_G(\Gamma)$ and $\Gamma_H := \pi_H(\Gamma)$. We then define a map $\tau: \Gamma_G \to H$ as $\tau := \pi_H \circ (\pi_G|_\Gamma)^{-1}$. Note that the image of $\tau$ is precisely $\Gamma_H$; in the abelian case this map is sometimes called the ``$*$-map''. Secondly, we denote by $Y := Y(G, H, \Gamma) := (G\times H)/\Gamma$ the associated compact homogeneous space and by $\pi_Y: G\times H \to Y$ the canonical projection. Inside $G \times H$ we can choose a compact subset $\mathcal F$ with dense interior $\mathcal F^o$ such that $\Gamma \mathcal F = G \times H$ and such that $\mathcal F^o$ is mapped homeomorphically to a dense open subset of $Y$. (For example, the Voronoi cell of the identity with respect to any left-admissible metric on $G \times H$ has these properties.) We then set $\mathcal F_G := \pi_G(\mathcal F)$ and $\mathcal F_H := \pi_H(\mathcal F)$ and observe that these are compact subsets of $G$ and $H$ respectively. Thirdly, we choose Haar measures $m_G$ and $m_H$ on $G$ respectively $H$ in such a way that $m_G \otimes m_H(\mathcal F) = 1$.
Then the unique invariant probability measure $m_Y$ on $Y$ is given in terms of the projection $\pi: G\times H \to Y$ by
\[
m_Y(f) = m_G \otimes m_H((\pi \circ f) \cdot \chi_{\mathcal F}) \quad (f \in C_c(Y)).
\] 
Finally, we will always choose a bounded Borel section $s: Y \to G \times H$, $x \mapsto (s_G(x), s_H(x))$ with values in $\mathcal F$ such that $\sigma(\pi(g,h)) = (g,h)$ for all $(g,h) \in \mathcal F^o$. Then, by construction,
\[
s_* m_Y = \chi_{\mathcal F} \cdot m_G \otimes m_H
\]
\end{remark}
As a consequence of our special choice of section we obtain:
\begin{lemma}\label{FMeasures} In the notation of Remark \ref{ModelSetConventions}, let $\mu_G = (s_G)_*m_Y$ and $\mu_H := (s_H)_*m_Y$. Then there exist bounded measurable functions $\rho_G \in L^\infty(\mathcal F_G)$ and $\rho_H \in L^\infty(\mathcal F_H)$ such that
\[
\mu_G = \rho_G \; m_G \quad \text{and} \quad \mu_H = \rho_H \; m_H.
\]
\end{lemma}
\begin{proof} Given $g \in G$, write $\mathcal F_g := \{h \in H \mid \mathcal (g,h) \in \mathcal F\} \subset \mathcal F_H$. Then $\mu_G = \rho_G \; m_G$, where $\rho_G(g) = m_H(\mathcal F_g)$. By definition we have $\rho_G(g) = 0$ unless $g \in \mathcal F_G$, and since $\rho_G(g) \leq m_H(\mathcal F_H)$ the function $\rho_G$ is bounded. The argument for $\mu_H$ follows by reversing the roles of $G$ and $H$.
\end{proof}
\begin{definition} Let  $(G, H, \Gamma)$ be a cut-and-project scheme. Given a compact subset $W \subset H$, the subset
\[
\Lambda(G, H, \Gamma, W) := \pi_G(\Gamma \cap (G \times W)) \subset G
\]
is called a \emph{weak model set} and $W$ is called its \emph{window}. 
\end{definition}
In terms of the map $\tau: \Gamma_G \to H$ from Remark \ref{ModelSetConventions} we have $\Lambda(G, H, \Gamma, W) = \tau^{-1}(W)$.
\begin{definition} A weak model set $\Lambda$ is called a \emph{model set} if its window has non-empty interior. It is called \emph{uniform} if the underlying cut-and-project scheme is uniform and \emph{symmetric} if it satisfies $\Lambda = \Lambda^{-1}$.
\end{definition}
The relation to approximate lattices is given by the following result.
\begin{proposition}[{\cite[Prop. 2.13]{BH}}] Every symmetric uniform model set in $G$ which contains the identity is a uniform approximate lattice in $G$. \qed
\end{proposition}
We now provide a condition which ensures that this uniform approximate lattice is strong.
\begin{definition}
Let $(G, H, \Gamma)$ be a cut-and-project scheme and let $W \subset H$ be compact. We say that $W$ is \emph{$\Gamma$-regular} if it is Jordan-measurable with dense interior, aperiodic (i.e., ${\rm Stab}_H(W) = \{e\}$) and satisfies $\partial W \cap \Gamma_H = \emptyset$. In this case the associated model set $\Lambda(G, H, \Gamma, W)$ is called a \emph{regular model set}.
\end{definition}
The following theorem summarizes basic results on \emph{regular} symmetric uniform model sets containing the identity.
\begin{theorem}[{\cite[Thm. 1.1]{BHP1}}]\label{ModelSetBackground}
Let $\Lambda = \Lambda(G, H, \Gamma, W)$ is a regular symmetric uniform model set containing the identity. Then the following hold:
\begin{enumerate}[(i)]
\item  $\Lambda$ is a strong approximate lattice in $G$, and $X_\Lambda$ admits a unique $G$-invariant measure $\nu$.
\item There exists a unique continuous $G$-equivariant surjection $\iota: X_\Lambda \to Y:=(G\times H)/\Gamma$ mapping $\Lambda$ to $(e,e)\Gamma$, which  induces a probability-measure preserving isomorphism
\[
(X_\Lambda, \nu) \to (Y, m_Y)
\]
of measurable $G$-spaces.\qed
\end{enumerate}
\end{theorem}
\begin{remark} Let $\Lambda = \Lambda(G, H, \Gamma, W)$ be as in the theorem. Since the projection of $\Gamma$ onto the first factor is injective, it induces a group isomorphism $\Gamma \cong \Lambda^\infty$. Under this isomorphism the set $\Lambda$ corresponds to the subset $\Gamma_W := \Gamma \cap (G \times W)$ of $\Gamma$. In this sense the approximate group $(\Lambda, \Lambda^\infty)$ is isomorphic to the approximate group $(\Gamma_W, \Gamma)$.
\end{remark}
\begin{definition} A model set $\Lambda = \Lambda(G, H, \Gamma, W)$ is said to have \emph{large window} if $W \supset \mathcal F_H$.
\end{definition}
Note that every (symmetric, regular) model set is contained in a (symmetric, regular) model set with large window. The following technical result will be used to provide an alternative model for cocycle induction for model sets.
\begin{proposition}\label{CompatibleSections}  Let $\Lambda = \Lambda(G, H, \Gamma, W)$ be a uniform model set with large window. Then for any bounded Borel section $s: Y \to G\times H$ as in Remark \ref{ModelSetConventions} there exists a bounded Borel section $\sigma: X_\Lambda \to G$ such that for almost all $P \in X_\Lambda$,
\[
\sigma(P) = \pi_G(s(\iota(P))).
\]
\end{proposition}
For the proof of the proposition we need an explicit almost everywhere defined inverse of the map $\iota: X_\Lambda \to Y$ from Theorem \ref{ModelSetBackground}.
\begin{lemma}\label{ModelSetX0Y0} Let $\Lambda = \Lambda(G, H, \Gamma, W)$ be a symmetric regular uniform model set and let $\iota: X_\Lambda \to Y$ be as in Theorem \ref{ModelSetBackground}. Then there exists a $\nu$-conull $G$-invariant Borel set $X_0 \subset X_\Lambda$ and a Haar-conull $G$-invariant Borel set $Y_0 \subset Y$ such that $\iota$ restricts to a $G$-equivariant measurable bijection $\iota_0: X_0 \to Y_0$ with inverse given by $\iota_0^{-1}((g,h)\Gamma) = g\tau^{-1}(h^{-1}W)$.
\end{lemma}
\begin{proof} The first statement is contained in \cite[Thm. 3.1 and Thm. 3.4]{BHP1}. We first claim that for every $P \in X_0$ and every $g_P \in P$ there exists $h_P \in H$ such that
\begin{equation}\label{Transversal1}
\iota_0(P) = (g_P, h_P)\Gamma \quad \text{and} \quad P = g_P.\tau^{-1}(h_P^{-1}W).
\end{equation}
To prove the claim, we fix $P \in X_0$ and $g_P \in P$. It is established in \cite[Prop. 2.10 and Sec. 2.5]{BHP1} that  $g_P^{-1}P$ is contained in the intersection of $X_0$ with the so-called canonical transversal $\mathcal T$. It then follows from \cite[Thm. 3.1 (ii) and (iii)]{BHP1} that there exists $h_P \in H$ such that $\iota(g_P^{-1}P) = (e, h_P)\Gamma$ and $\iota^{-1}((e, h_P)\Gamma) = \tau^{-1}(h_P^{-1}W)$ for some $h_P \in H$. By $G$-equivariance of $\iota_0$ and $\iota_0^{-1}$ we have
\[
\iota_0(P) = \iota_0(g_Pg_P^{-1}P) = g_P \iota_0(g_P^{-1}P) = g_P(e, h_P) = (g_P, h_P)\Gamma
\]
and $P = \iota_0^{-1}((g_p, h_P) \Gamma) = g_p.\iota_0^{-1}((e, h_P)\Gamma) = g_P.\tau^{-1}(h_P^{-1}W)$. This finishes the proof of \eqref{Transversal1}.\\

Now let $(g,h)\Gamma \in Y_0$ and set $P := \iota_0^{-1}((g,h)\Gamma) \in X$. Choose $g_P \in P$ and let $h_P \in H$ such that \eqref{Transversal1} holds. Then $(g,h)\Gamma = \iota_0(P) = (g_P, h_P)\Gamma$, hence there exists $(\gamma, \gamma^*) \in \Gamma$ such that $(g,h) = (g_P\gamma, h_P\gamma^*)$, and hence
\[
P = g_P.\tau^{-1}(h_P^{-1}W) = g_P \, \pi_G((\gamma, \gamma^*)(\Gamma \cap (G \times (\gamma^*)^{-1}h_P^{-1}W))) = g\tau^{-1}(h^{-1}W).\qedhere
\]
\end{proof}

\begin{proof}[Proof of Proposition \ref{CompatibleSections}] Let $X_0$ and $Y_0$ as in Lemma \ref{ModelSetX0Y0} and let $P \in X_0$. Abbreviate $(g,h) := s(\iota(P))$. By construction, $(g,h) \in \mathcal F$ and hence $h \in \mathcal F_H$. By Lemma \ref{ModelSetX0Y0} we have 
\[P = g \; \pi_G(\Gamma \cap (G \times h^{-1}W))\]
Since $h \in \mathcal F_H \subset W$ the set $G \times h^{-1}W$ contains the identity, hence $g \in P$. Thus 
\begin{equation}\label{sigmavss}
g = \pi_G(s(\iota(P))) \in P \quad \text{for all }P \in X_0.
\end{equation}
We now fix an arbitrary bounded Borel section $\sigma': X_\Lambda \to G$ and define
\[
\sigma: X_\Lambda \to G, \quad \sigma(P) = \left\{\begin{array}{ll}  \pi_G(s(\iota(P))), & \text{if }P \in X_0,\\
\sigma'(P) & \text{if }P \not \in X_0.
 \end{array}\right.
\]
It follows from \eqref{sigmavss} and the fact that $\sigma'$ is a section, that $\sigma$ is a section. Moreover, $\sigma$ is Borel, since $s$ and $\sigma'$ are and since $X_0$ and its complement are Borel sets. It is bounded, since $s$ and $\sigma'$ is bounded. Finally, $\sigma(P) = \pi_G(s(\iota(P)))$ for all $P \in X_0$, hence for almost all $P \in X_\Lambda$.
\end{proof}

\section{Affine isometric actions and $L^p$-induction}
\label{affineisom}

\subsection{Affine isometric actions on Banach spaces and (weak) quasi-cocycles}

In the sequel, all Banach spaces are assumed to be defined over the field of real numbers. Given a Banach space $(E, \|\cdot\|)$ we denote by ${\rm Is}(E)$ the corresponding isometry group. By the Mazur-Ulam theorem we have ${\rm Is}(E) = O(E) \ltimes E$, where $O(E)$ denotes the orthogonal group of $(E, \|\cdot\|)$ (i.e., the group of linear isometries), and $E$ acts on itself by translations. If $\Gamma$ is a group, then a homomorphism $\rho: \Gamma \to {\rm Is}(E)$ is called an \emph{affine isometric action} of $\Gamma$ on $E$. Every such action is of the form 
\[\rho(g).v = \pi(g).v + b(g), \quad (g \in \Gamma, v \in E)\]
 where $\pi: \Gamma \to O(E)$ is a homomorphism, and $b: \Gamma \to E$ is a $1$-cocycle with respect to $\pi$ in the sense that
\begin{equation}\label{1cocycle}
b(gh) = b(g) + \pi(g)b(h) \quad (g,h \in \Gamma).
\end{equation}
We then write $\rho = (\pi, b)$ and refer to $\pi$ and $b$ as the \emph{linear part} of $\rho$, respectively the \emph{cocycle} defined by $\rho$. For later use we record that if $b$ is a cocycle then $b(e) = 0$ (since $b(e) = b(e) + \pi(e) b(e) = 2b(e)$) and hence for all $g \in G$ we have
\[
 0 = b(gg^{-1}) = b(g) + \pi(g)b(g^{-1}),
\] 
i.e.,
\begin{equation}\label{bInverse}
b(g^{-1}) = -\pi(g)^{-1}b(g).
\end{equation}
In the remainder of this article we will only consider uniformly convex Banach spaces. If $G$ is a topological group and $\pi: G \to O(E)$ is a homomorphism, then the action map $G \times E \to E$ is jointly continuous if and only if $\pi$ is weakly continuous (equivalently, strongly continuous) if and only if the orbit maps of $\pi$ are continuous (as follows e.g. \cite[Lemma 2.4]{BFGM} since $E$ is uniformly convex and hence superreflexive). In this case we call $\pi$ a \emph{continuous orthogonal representation}.
\begin{definition} An affine isometric action $\rho = (\pi, b)$ of a topological group $g$ on a uniformly convex Banach space $E$ is called \emph{continuous} if $\pi$ is a continuous orthogonal representation and $b$ is a continuous cocycle.
\end{definition}
From now on let $\mathcal E$ be a class of uniformly convex separable Banach spaces. In particular, given $1 < p < \infty$ we will be interested in the class of $L^p$-spaces, i.e.\ Banach spaces $E$ which are isometrically isomorphic to $L^p(Y, \mu)$ for some $\sigma$-finite Borel measure $\mu$ on a standard Borel space $Y$. We insist on separability in order to ensure that ${\rm Is}(E)$ and $O(E)$ are Polish groups with the topology of pointwise convergence (a.k.a. strong operator topology) for every $E \in \mathcal E$.
\begin{definition}
If $G$ is a lcsc group and $\rho = (\pi, b)$ is a continuous affine isometric action on some $E \in \mathcal E$, then we refer to $\rho$ as an \emph{affine $\mathcal E$-action of $G$} and to $b$ as an \emph{$\mathcal E$-cocycle} on $G$.
\end{definition}
If $\mathcal E$ is the class of all $L^p$-spaces we obtain in particular the notion of an affine $L^p$-action and $L^p$-cocycle. Note that our affine actions will always implicitly assumed to be continuous. 

For the class $\mathcal E = L^2$ of separable Hilbert spaces the following weakenings of the notion of an $L^2$-cocycle were introduced by Ozawa in his work on Property (TTT), respectively by Burger--Monod in their work on Property (TT). Here, given $E \in \mathcal E$, the group $O(E)$ is equipped with the Borel structure associated with the strong (equivalently, the weak) operator topology.
\begin{definition} Let $G$ be a lcsc group and $E \in \mathcal E$. Let $\pi: G \to O(E)$ be a Borel map (not necessarily a homomorphism) and let $b: G \to E$ be a Borel map which is locally bounded in the sense that for every compact subset $K \subset E$,
\[
\sup_{g \in K}\|b(g)\| < \infty.
\]
\begin{enumerate}[(i)]
\item We say that $b$ is a \emph{weak $\mathcal E$-quasi-cocycle} (wq-$\mathcal E$-cocycle for short) with respect to $\pi$ if
\begin{equation}
\label{def_Db}
D(b) := \sup_{g_1, g_2 \in G} \|b(g_1g_2)-b(g_1) - \pi(g_1)b(g_2)\| < \infty.
\end{equation}
Then the pair $(\pi, b)$ is called a \emph{wq-$\mathcal E$-pair} for $G$, and $E$ is called the \emph{underlying space}.
\item If $(\pi, b)$ is a wq-$\mathcal E$-pair and $\pi$ is moreover a homomorphism, then $b$ is called a \emph{quasi-cocycle} and $(\pi, b)$ is then called a \emph{quasi-$\mathcal E$-pair}.
\end{enumerate}
In particular, we obtain the notions of an $L^p$-quasi-cocycle and wq-$L^p$-cocycle.
\end{definition} 
By definition, every $L^p$-cocycle is an $L^p$-quasi-cocyle, and every $L^p$-quasi-cocycle is a wq-$L^p$-cocycle. In particular, all results concerning wq-$L^p$-cocycles below apply to $L^p$-quasi-cocycles and $L^p$-cocycles.

\begin{remark} Our definition of a quasi-cocycle follows Ozawa \cite{Oz11}. Burger and Monod \cite{BM1} require in addition that $b$ be continuous. However, it is well-known\footnote{Since we could not locate a precise reference in the literature, we include a proof in Appendix \ref{AppendixCocycle}.} that every Borel quasi-cocycle is at uniformly bounded distance from a continuous quasi-cocycle, hence this difference in definition does not affect the notion of Property (FF$\mathcal E$) defined below.
\end{remark}

\subsection{$L^p$-induction for strong uniform approximate lattices}
For the rest of this section we consider the following setting: Let $G$ be a lcsc group, let $\Lambda \subset G$ be a strong uniform approximate lattice with enveloping group $\Lambda^\infty$, and denote by $X = X_\Lambda$ the hull of $\Lambda$. We fix a left-admissible metric $d$ on $G$ and given $g\in G$ we set $\|g\| := d(g,e)$. Let $R > r > 0$ be Delone parameters of $\Lambda$ with respect to $d$. Using Proposition \ref{ExSection} we choose a Borel section $s: X_\Lambda \to G$ which takes values in $B_{2R}(e)$ and denote by $\beta = \beta_s: G \times X_\Lambda \to \Lambda^2$ the associated cocycle.

If $\Lambda$ happens to be a uniform \emph{lattice} in $G$, then every affine $L^p$-action of $\Lambda$ induces an affine $L^p$-action of $G$ (see e.g. \cite{Sh}), and we would like to generalize this construction to the case at hand. In fact, it is natural to discuss induction in the wider context of wq-$L^p$-cocycles. In the present setting the situation is complicated by the fact that there may be more than one $G$-invariant measure on $X$. As we will see in Example \ref{MeasureDependence} below, different choices of measures on $X$ will lead to substantially different induction procedures. 

We now fix a $G$-invariant measure $\nu$ on $X$ and proceed to define an induction operation depending on $\nu$. For this let $Y$ be a standard Borel space, $\mu$ a $\sigma$-finite Borel measure on $Y$ and $E := L^p(Y, \mu)$ for some $p \in (1, \infty)$. We then denote by 
\[
\widehat{E} := L^p(X_\Lambda, \nu; E)
\]
the space of equivalence classes of Bochner $p$-integrable $E$-valued functions on $(X_\Lambda, \nu)$. Explicitly, a function $f: X_\Lambda \to E$ represents a class in $\widehat{E}$ if and only if the map $\bar{f}: X_\Lambda \times Y \to \C$ given by
\[
\bar f(P, y) :=  f(P)(y)
\]
represents a class in $L^p(X_\Lambda \times Y, \nu\otimes \mu)$. In particular, we see from the isometric isomorphism
\[
\widehat{E} \xrightarrow{\cong} L^p(X_\Lambda \times Y, \nu\otimes \mu), \quad [f] \mapsto [\bar f]
\]
that $\widehat{E}$ is again an $L^p$-space.

\begin{remark} If $E$ is an arbitrary uniformly convex separable Banach space, then the Banach space $\widehat{E} := L^p(X_\Lambda, \nu; E)$ of equivalence classes of Bochner $p$-integrable $E$-valued functions on $(X_\Lambda, \nu)$ can always be defined. We say that a class $\mathcal E$ of uniformly convex separable Banach spaces is \emph{$L^p$-closed} if $\widehat{E} \in \mathcal E$ for every $E \in \mathcal E$. By the previous remark, the class of $L^p$-spaces itself is $L^p$-closed. While we are mainly interested in the class of $L^p$-spaces here, most of what we say below can be established for an arbitrary $L^p$-closed class $\mathcal E$ of uniformly convex separable Banach spaces.
\end{remark}

\begin{proposition}\label{InductionWellDefined} Let $1 < p < \infty$ and let $\mathcal E$ be an $L^p$-closed class of uniformly convex separable Banach spaces (for example, the class of $L^p$-spaces). Let $(\pi, b)$ be a wq-$\mathcal E$-pair for $\Lambda^\infty$ with underlying space $E \in \mathcal E$ and let $\widehat{E} := L^p(X_\Lambda, \nu; E)$.
\begin{enumerate}[(i)]
\item There are well-defined Borel maps 
\begin{eqnarray*}
\widehat{\pi}: G \to U(\widehat{E}), & \widehat{\pi}(g)f(P) &:= \quad  \pi(\beta(g^{-1}, P)^{-1})f(g^{-1}P),\\
\widehat{b}: G \to \widehat{E}, & \widehat{b}(g) &:= \quad b (\beta(g^{-1}, P)^{-1}),
\end{eqnarray*}
and $(\widehat{\pi}, \widehat{b})$ is a wq-$\mathcal E$-pair for $G$ with underlying space $\widehat{E} \in \mathcal E$ satisfying $D(\widehat{b}) \leq D(b)$.
\item If $(\pi, b)$ is a quasi-$\mathcal E$-pair, then so is $(\widehat{\pi}, \widehat{b})$.
\item If $(\pi, b)$ is an affine $\mathcal E$-action, then so is $(\widehat{\pi}, \widehat{b})$.
\end{enumerate}
\end{proposition}
\begin{proof} (i) $\widehat{\pi}$ maps $G$ into $U(\widehat{E})$ since $\pi$ is orthogonal and is clearly a Borel map. As for well-definedness of $\widehat{b}$ we have to show that $\widehat{b}(g)$ is $p$-integrable for every $g \in G$. Recall from \eqref{Xglambda} the definition of the sets $X_\Lambda(g, \lambda)$ for $g \in G$ and $\lambda \in \Lambda^2$. Since $\beta$ takes values in $\Lambda^2$ we have
\[
X_\Lambda = \bigsqcup_{\lambda \in \Lambda^2} X_\Lambda(g, \lambda).
\]
for every fixed $g \in G$. We deduce that
\begin{eqnarray*} \|\widehat{b}(g)\|^p &=& \int_{X_{\Lambda}} \|\widehat{b}(g)(P)\|^p \, d\nu(P) \quad = \quad  \int_{X_{\Lambda}} \|b(\beta(g^{-1}, P)^{-1})\|^p \,d\nu(P)\\
&=& \sum_{\lambda \in \Lambda^2} \|b(\lambda)\|^p \nu(X_\Lambda(g, \lambda)).
\end{eqnarray*}
By Lemma \ref{HaagerupTriangleInequality} we have $\nu(X_\Lambda(g, \lambda)) = 0$ unless $\|\lambda\| \leq \|g\| +4 R$. Since $\Lambda$ is uniformly discrete in $G$, it follows that all but finitely many summands are $0$. This shows that $\|\widehat{b}(g)\| < \infty$, whence $\widehat{b}(g) \in \widehat{E}$ for all $g \in G$, and the map $\widehat{b}$ is evidently Borel. It remains to show that $\widehat{b}$ is a wq-cocycle with respect to $\widehat{\pi}$. To this end we observe that for all $g_1, g_2 \in G$ and $P \in X_\Lambda$ we have
\[
\beta(g_2^{-1}g_1^{-1}, P)^{-1}) \quad = \quad \beta(g_1^{-1}, P)^{-1} \beta(g_2^{-1}, g_1^{-1}P)^{-1}),
\]
and hence
\begin{eqnarray*}
 \|\widehat{b}(g_1g_2)-\widehat{b}(g_1) - \widehat{\pi}(g_1)\widehat{b}(g_2)\|
& =& \| b(\beta(g_1^{-1}, P)^{-1} \beta(g_2^{-1}, g_1^{-1}P)^{-1})- b (\beta(g^{-1}, P)^{-1}) \\&&- \pi(\beta(g^{-1}, P)^{-1})b( \beta(g_2^{-1}, g_1^{-1}P)^{-1})\|
\quad \leq\quad D(b),
\end{eqnarray*}
which shows that $D(\widehat{b}) \leq D(b)$ and establishes (i).

(ii) It remains to show only that if $\pi$ is a homomorphism, then so is $\widehat{\pi}$. Now by \eqref{betaCocycle} we have for all $g_1,g_2 \in G$ and $P \in X_\Lambda$,
\begin{eqnarray*}
\widehat{\pi}(g_1g_2)f(P) &=& \pi(\beta(g_2^{-1}g_1^{-1}, P)^{-1})f(g_2^{-1}g_1^{-1}P)\\
&=& \pi(\beta(g_1^{-1}, P)^{-1})\pi(\beta(g_2^{-1}, g_1^{-1}P)^{-1}) f(g_2^{-1}(g_1^{-1}P))\\
&=& \pi(\beta(g_1^{-1}, P)^{-1})\widehat{\pi}(g_2)f(g_1^{-1}P)\\
&=& \widehat{\pi}(g_1)\widehat{\pi}(g_2)f(P),
\end{eqnarray*}
and hence $\widehat{\pi}(g_1g_2) = \widehat{\pi}(g_1)\widehat{\pi}(g_2)$.

(iii) Assume that $(\pi, b)$ is an affine $\mathcal E$-action, i.e., a quasi-$\mathcal E$-pair with $D(b) = 0$. By (ii), this implies that $(\widehat{\pi}, \widehat{b})$ is a quasi-$\mathcal E$-pair, and by (i) we have $D(\widehat{b}) \leq D(b) = 0$, hence $b$ is a cocycle. Now the $\widehat{\rho} = (\widehat{\pi}, \widehat{b}): G \to {\rm Is}(\widehat{E})$ is Borel. It the follows from \cite{Pe} that $\widehat{\rho}$ is continuous, hence defines an affine $\mathcal E$-action of $G$.
\end{proof}
\begin{definition} In the situation of Proposition \ref{InductionWellDefined} the pair $\widehat{\rho} = (\widehat{\pi}, \widehat{b})$ is said to be \emph{$(\nu, s)$-$L^p$-induced} from $(\pi, b)$ and we write
\[
(\nu,s,p){\rm -Ind}_{(\Lambda, \Lambda^\infty)}^G (\pi, b) := (\widehat{\pi}, \widehat{b}).
\]
If $b$ is a wq-$\mathcal E$-cocycle (respectively an $\mathcal E$-quasicocycle or an $\mathcal E$-cocycle), then $\widehat{b}$ is called the $(\nu, s)$-$L^p$-induced wq-$\mathcal E$-cocycle (respectively the $(\nu, s)$-$L^p$-induced $\mathcal E$-quasicocycle or the $(\nu, s)$-$L^p$-induced $\mathcal E$-cocycle).
\end{definition}
\begin{remark} $L^p$-induction applies in particular to the case where $\Lambda = \Lambda^\infty$ is actually a uniform lattice in $G$, and in this cases we recover the classical constructions. In the sequel, when dealing with induction from uniform lattices we will thus use the same notations as introduced in the approximate lattice case above. 
\end{remark}

\subsection{Dependence on the section and the measure}
We keep the notation of the previous subsection. In particular, $1 < p < \infty$ and $\mathcal E$ denotes an $L^p$-closed class of uniformly convex separable Banach spaces (for example, the class of $L^p$-spaces). 

\subsubsection{Independence of the section}
It turns out that the dependence of $(\nu,s){\rm -Ind}_{(\Lambda, \Lambda^\infty)}^G (\pi, b)$ on the section $s$ is inessential in the following sense. Let $s_1, s_2: X_\Lambda \to G$ be bounded Borel sections and recall from Lemma \ref{ChangeOfSection} that
\[
\beta_{s_2}(g^{-1}, P)^{-1} = u(P)^{-1}\beta_{s_1}(g^{-1}, P)^{-1}u(g^{-1}P).
\]
where $u: X_\Lambda \to \Lambda^2$ is given by $u(P) :=  s_1(P)^{-1}s_2(P)$. Now let $(\pi, b)$ be be a wq-$L^p$-pair for $\Lambda^\infty$ with underlying space $E$, and set
\[
(\widehat{\pi}_j, \widehat{b}_j) := (\nu, s_j){\rm -Ind}_{(\Lambda, \Lambda^\infty)}^G (\pi, b)\quad (j \in \{1,2\}).
\]
Both pairs have the same underlying space $\widehat{E} = L^p(X_\Lambda, E)$, and we define an isometric isomorphism $U: \widehat{E} \to \widehat{E}$ by
\[
Uf(P) := \pi(u(P)) f(P) + b(u(P)).
\]
If we assume that $(\pi, b)$ is an affine $\mathcal E$-action, then by \eqref{bInverse} we have
\[
\pi(u(P))b(u(P)^{-1})+b(u(P) = 0 \quad \text{for all }P \in X_\Lambda,
\]
and hence
\begin{eqnarray*}
(U \circ \widehat{\rho}_2(g))f(P) &=& \pi(u(P)) (\pi(\beta_{s_2}(g^{-1}, P)^{-1})f(g^{-1}P)+ \widehat{b}_2(g)) + b(u(P))\\
&=& \pi(u(P))\pi(u(P)^{-1} \beta_{s_1}(g^{-1}, P)^{-1} u(g^{-1}P))f(g^{-1}P)\\
&&+ \pi(u(P))b(u(P)^{-1} \beta_{s_1}(g^{-1}, P)^{-1} u(g^{-1}P)) + b(u(P))\\
&=&  \pi(\beta_{s_1}(g^{-1}, P)^{-1})\pi(u(g^{-1}P))f(g^{-1}P)\\
&&+ b(\beta_{s_1}(g^{-1}, P)^{-1}) + \pi(\beta_{s_1}(g^{-1}, P)^{-1}) b(u(g^{-1}P))\\
&=& \widehat{\pi}_1(g)(\pi(u(P)) f(P) + b(u(P))) +  \widehat{b}_1(g)(P)\\
&=& \widehat{\pi}_1(g)(U f)(P) + \widehat{b}_1(g)(P)\\
&=&(\widehat{\rho_1}(g) \circ U)f(P).
\end{eqnarray*}
This shows that if $(\pi, b)$ is an affine $\mathcal E$-action, then $(\widehat{\pi}_1, \widehat{b}_1)$ and $(\widehat{\pi}_2, \widehat{b}_2)$ are intertwined by the isometric isomorphism $U$, hence define isomorphic affine actions.

If $(\pi, b)$ is only a quasi-$\mathcal E$-pair, then $U$ still intertwines the orthogonal representations $\widehat{\pi}_1$ and $\widehat{\pi}_2$, and the quasi-cocycles $\widehat{b}_1$ and $\widehat{b}_2$ are intertwined by $U$ up to a bounded error.

Finally, if $(\pi, b)$ is merely a wq-$\mathcal E$-pair, then there is no control about  $\widehat{\pi}_1$ and $\widehat{\pi}_2$, but the induced wq-$\mathcal E$-cocycles $\widehat{b}_1$ and $\widehat{b}_2$ are still intertwined by $U$ up to a bounded error.

It follows that properties of $\widehat{b}$ which are stable under bounded pertubations, such as boundedness or properness, are independent of the choice of section used to define it. In view of this essential independence of the section we will often write
\[
(\nu,p){\rm -Ind}_{(\Lambda, \Lambda^\infty)}^G (\pi, b) := 
 (\nu,s,p){\rm -Ind}_{(\Lambda, \Lambda^\infty)}^G (\pi, b) 
\]
and suppress the choice of section $s$ from our notation.

\subsubsection{Dependence on the measure}
The following example shows that even in the case of an affine $L^p$-action $(\pi, b)$ the induced representation $(\nu,p){\rm -Ind}_{(\Lambda, \Lambda^\infty)}^G (\pi, b)$ does depend in an essential way on the choice of invariant measure $\nu$ on the hull:
\begin{example}\label{MeasureDependence} 
Let $G = (\bR,+)$ and suppose that $\Lambda \subset G$ is a uniform approximate lattice with $\Lambda^\infty = \bZ$ such that the hull $X_\Lambda$ contains 
the sets $2\bZ$ and $3\bZ$, whence the $G$-hulls $X_2 := X_{2\bZ}$ and $X_3 := X_{3\bZ}$. (An explicit construction of such a set $\Lambda$ is given e.g.\ in \cite[Example 4.15]{BH}.) We note that $X_2$ and $X_3$ admit unique $G$-invariant probability measures $\nu_2$ and $\nu_3$, which we can view as probability measures on $X_\Lambda$ as well. For every affine $L^p$-action $(\pi, b)$ of $\Lambda^\infty$ and $q \in \{2,3\}$ we have
\[
(\widehat{\pi}_q, \widehat{b}_q) := \nu_q{\rm -Ind}_{(\Lambda, \Lambda^\infty)}^G (\pi, b) \cong {\rm Ind}_{q\Z}^G(\pi|_{q\Z}, b|_{q\Z}).
\]
In particular, if $(\pi, b)$ is realized on an $L^p$-space $E$ and $T := \pi(1) $, then $\nu_q{\rm -Ind}_{(\Lambda, \Lambda^\infty)}^G (\pi, b)$ can be realized on the space
\[
\widehat{E}_q := \{f \in L^p(\R; E) \mid f(x + q) = T^q f(x) \},
\]
with $\widehat{\pi}_q(g)f(x) := f(x-g)$. In general, the orthogonal representations $\widehat{\pi}_2$ and $\widehat{\pi}_3$ are not equivalent.
For instance, consider the case when $E = \bR$ and $Tv = -v$ for $v \in E$. If $\widehat{E}_2$ and $\widehat{E}_3$ were isomorphic (as
orthogonal representations) via some linear isomorphism $S : \widehat{E}_2 \ra \widehat{E}_3$, then a straightforward calculation shows that
$S(f + \widehat{\pi}_2(1) f) = 0$, and thus every $f \in \widehat{E}_2$ would satisfy $f(x-1) = -f(x)$, which is clearly not the case for the 
element $f \equiv 1$ in $\widehat{E}_2$. Thus, in general,
\[
\nu_2{\rm -Ind}_{(\Lambda, \Lambda^\infty)}^G (\pi, b) \quad \not \cong \quad \nu_3{\rm -Ind}_{(\Lambda, \Lambda^\infty)}^G (\pi, b).
\]
\end{example}

\subsection{$L^p$-induction for model sets}
We illustrate our construction of $L^p$-induction for the case of uniform regular symmetric model sets. Throughout this subsection let  $(G, H, \Gamma)$ be a uniform cut-and-project scheme and let $\Lambda = \Lambda(G, H, \Gamma, W)$ be a symmetric regular uniform model set. We will use the notations and conventions introduced in Remark \ref{ModelSetConventions}. In particular, we set $Y := (G\times H)/\Gamma$. We recall that the map $\pi_G$ induces an isomorphism 
\[
\pi_G: (\Gamma_W, \Gamma) \to (\Lambda, \Lambda^\infty)
\]
of approximate groups, where $\Gamma_W := \Gamma \cap (G \times W)$.

Now let $\mathcal E$ denote an $L^p$-closed class of uniformly convex separable Banach spaces and consider an affine isometric action $\rho_0 = (\pi_0, b_0)$ of $\Lambda^\infty$ on some $E \in \mathcal E$. Via the isomorphism $\pi_G$ this induces an affine $L^p$-action $\rho := \rho_0 \circ \pi_G = (\pi, b)$ of $\Gamma$. Since $\Gamma$ is a uniform lattice in $G \times H$, we can induce this action to obtain an affine action ${\rm Ind}_\Gamma^{G \times H} (\pi, b)$ of $G \times H$ on the Banach space $\widehat{E} := L^p((G\times H)/\Gamma; E)$.
\begin{proposition}\label{ModelModel} Assume that $\Lambda$ has large window and denote by $\nu$ the unique $G$-invariant measure on $X_\Lambda$. Then
\begin{equation}\label{ModelIsoModel}
\nu{-\rm Ind}_{(\Lambda, \Lambda^\infty)}^G (\pi_0, b_0) \cong \left.\left({\rm Ind}_\Gamma^{G \times H} (\pi, b)\right)\right|_G.
\end{equation}
\end{proposition}
\begin{proof} Firstly, let $\iota: (X_\Lambda, \nu) \to (Y, m_Y)$ denote the measurable isomorphism from Theorem \ref{ModelSetBackground}. 
Secondly, let $s: Y \to G \times H$ be a bounded Borel section chosen as in Remark \ref{ModelSetConventions} and denote by $(\widehat{\pi}, \widehat{b})$ the representation ${\rm Ind}_\Gamma^{G \times H} (\pi, b)$ on $\widehat{E}$ defined by means of this section. By Proposition \ref{CompatibleSections} we can choose a bounded Borel section $\sigma: X_\Lambda \to G$ such that for almost all $P \in X_\Lambda$,
\[
\sigma(P) = \pi_G(s(\iota(P))).
\]
Denote by $(\widetilde{\pi},\widetilde{b})$ the model of $\nu{-\rm Ind}_{(\Lambda, \Lambda^\infty)}^G (\pi_0, b_0)$ defined on $\widetilde{E} := L^p(X_\Lambda, \nu; E)$ by means of the section $\sigma$. Then $\iota$ induces an isomorphism
\[
\iota^*: \widehat{E} \to \widetilde{E}, \quad \iota^*f := f\circ \iota.
\]
We claim that this isomorphism intertwines the action of $G \times \{e\}$ via $(\widehat{\pi}, \widehat{b})$ on $\widehat{E}$ with the action of $G$ on $\widetilde{E}$ via $(\widetilde{\pi}, \widetilde{b})$. Towards the proof of the claim we first observe that for $g \in G$ and $x \in X_\Lambda$ we have
\begin{eqnarray*}
\beta_\sigma(g^{-1}, x)^{-1} &=& \sigma(x)^{-1} g\sigma(gx) \quad = \quad \pi_G(s(\iota(x)))^{-1}g\pi_G(s(\iota(gx)))\\
&=& \pi_G(s(\iota(x))^{-1}(g,e)s((g,e)\iota(x)))\\
&=&\pi_G(\beta_s(g^{-1}, \iota(x))^{-1}),
\end{eqnarray*}
hence
\[
\pi_0(\beta_\sigma(g^{-1}, x)^{-1}) = \pi(\beta_s(g^{-1}, \iota(x))^{-1}) \quad \text{and} \quad b_0(\beta_\sigma(g^{-1}, x)^{-1}) = b(\beta_s(g^{-1}, \iota(x))^{-1}).
\]
This implies that
\begin{eqnarray*}
(\widetilde{\rho}(g) \circ \iota^*)f(x) &=& \pi_0(\beta_\sigma(g^{-1}, x)^{-1})f(\iota(g)^{-1}x) +  b_0(\beta_\sigma(g^{-1}, x)^{-1})\\
&=&  \pi(\beta_s(g^{-1}, \iota(x))^{-1})f((g,e)^{-1}\iota(x)) + b(\beta_s(g^{-1}, \iota(x))^{-1})\\
&=& (\iota^* \circ \widehat{\pi}(g,e))f(x),
\end{eqnarray*}
which establishes the claim and thereby finishes the proof.
\end{proof}

\begin{remark} 
Two remarks are in order:
\begin{enumerate}[(i)]
\item Proposition \ref{ModelModel} shows in particular that the $L^p$-induced affine action $\nu{-\rm Ind}_{(\Lambda, \Lambda^\infty)}^G (\pi_0, b_0)$ can be extended to an affine $\mathcal E$-action of the lcsc group $G \times H$. The additional flexibility coming from the $H$-action will turn out to be useful e.g. in establishing various Kazhdan properties of model sets.
\item The right-hand side of \eqref{ModelIsoModel} makes sense for every model set $\Lambda$, regardless of whether the hull $X_\Lambda$ admits a $G$-invariant measure or whether $\Lambda$ has large window. This allows us to define induction also for model sets, which are not strong approximate lattice.
\end{enumerate}
\end{remark}

\section{Analytic properties of approximate groups}
\label{analyticprops}

\subsection{Property (T) and its relatives}
We spell out the definitions of the various Kazhdan-type properties which we investigate in the sequel. Throughout this section, $1 < p < \infty$ and $\mathcal E$ denotes an $L^p$-closed class of uniformly convex separable Banach spaces (for example, the class of $L^p$-spaces). 
\begin{definition} Let $G$ be a lcsc group and let $A$ be subset of $G$ and $p \in (1, \infty)$.
\begin{enumerate}[(i)]
\item $G$ has Property (F$\mathcal E$) relative to $A$ if every $\mathcal E$-cocycle on $G$ is bounded on $A$.
\item $G$ has Property (FF$\mathcal E$) relative to $A$ if every $\mathcal E$-quasi-cocycle on $G$ is bounded on $A$.
\item $G$ has Property (FFF$\mathcal E$) relative to $A$ if every weak $\mathcal E$-quasi-cocycle on $G$ is bounded on $A$.
\end{enumerate}
If $A = G$ then we simply say that $G$ has Property (F$\mathcal E$), (FF$\mathcal E$) or (FFF$\mathcal E$) respectively. 
In particular, this defines relative and absolute versions of Property (F$L^p$), (FF$L^p$) and (FFF$L^p$). If $p=2$, then we also write ``$H$'' (for Hilbert space) instead of ``$L^2$''. \end{definition}
\begin{remark}\label{RemarkKazhdan}
\begin{enumerate}[(i)]
\item In the sequel we refer to the Properties (F$\mathcal E$), (FF$\mathcal E$) and (FFF$\mathcal E$) as \emph{Kazhdan-type properties} with respect to $\mathcal E$. Property (F$\mathcal E$) is equivalent to $A$ having bounded orbits in every affine $\mathcal E$-action of $G$.
\item If $G$ enjoys some Kazhdan-type property, then $(G, A)$ also enjoys this property for every subset $A\subset G$. 
\item If $A$ is actually a \emph{subgroup}, then Property (F$\mathcal E$) is equivalent to $A$ having a fixpoint in every affine $\mathcal E$-action of $G$, hence the name. For general subsets, relative Property (F$\mathcal E$) is \emph{not} a fixpoint property (and hence a more appropriate term might be ``Property (B$L^p$)''). In fact, $A$ having a fixpoint in every affine $\mathcal E$-action of $G$ is equivalent to $G$ having Property (F$\mathcal E$) relative to the group generated by $A$, which is in general a much stronger property.
\item As remarked in the introduction, Property (FF$H$) was introduced by Burger and Monod \cite{BM1} under the name Property (TT), and Property (FFF$H$) was introduced by Ozawa \cite{Oz11} under the name Property (TTT). Cornulier \cite{Co1} was the first to systematically study Property (T) relative to arbitrary subsets. The extension to $L^p$-spaces (and more general classes of Banach spaces) goes back to Bader, Furman, Gelander and Monod \cite{BFGM}.
\end{enumerate}
\end{remark}
We now specialize to our case of interest:
\begin{definition} Let $\mathcal T$ be a Kazhdan-type property. An approximate group $(\Lambda, \Lambda^\infty)$ is said to have $\mathcal T$ if $\Lambda^\infty$ has $\mathcal T$ relative to $\Lambda$.
\end{definition}
\begin{remark} As a special case of Remark \ref{RemarkKazhdan} we observe that if $\Lambda^\infty$ enjoys a Kazhdan-type property $\mathcal T$, then so does $(\Lambda, \Lambda^\infty)$.  We will see in Example \ref{ExTnoT} below that the converse is not true.
\end{remark}

\subsection{Elementary constructions preserving Kazhdan-type properties}
The following facts follow straight from the definitions.
\begin{lemma} Let $\mathcal E$ be a class of uniformly convex separable Banach spaces. Let $G_1, G_2$ be lcsc groups, $A_1 \subset G_1$, $A_2 \subset G_2$ be subsets, and let $\mathcal T$ be a Kazhdan type property with respect to $\mathcal E$.
\begin{enumerate}[(i)]
\item If $G_1$ has $\mathcal T$, then $(G_1, A_1)$ has relative $\mathcal T$
\item If $A_1$ is finite, then $(G_1, A_1)$ has relative $\mathcal T$.
\item If $(G_1, A_1)$ and $(G_2, A_2)$ have relative $\mathcal T$, then so has $(G_1 \times G_2, A_1 \times A_2)$.
\item If $G_1 = G_2$ and $(G_1, A_1)$ and $(G_2, A_2)$ have relative $\mathcal T$, then so has $(G_1, A_1A_2)$.
\item If $F\subset G_1$ is finite and $(G_1, A_1)$ has relative $\mathcal T$, then so have $(G_1, A_1F)$ and $(G_1, FA_1)$.\qed
\end{enumerate}
\end{lemma}
Concerning passage to subpairs we observe:
\begin{lemma}\label{TUnderFiniteIndex} Let $\Gamma_1 < \Gamma_2$ be countable groups, $A_2 \subset \Gamma_2$, $A_1 \subset  \Gamma_1 \cap A_2$, , and let $\mathcal T$ be a Kazhdan type property relative to a class $\mathcal E$ of uniformly convex separable Banach spaces which is $L^p$-closed for some $p \in (1, \infty)$.
\begin{enumerate}[(i)]
\item If $(\Gamma_2, A_2)$ has relative $\mathcal T$ and $\Gamma_1$ has finite index in $\Gamma_2$, then $(\Gamma_1, A_1)$ has relative $\mathcal T$.
\item If $A_1$ is left-syndetic in $A_2$ and $(\Gamma_1, A_1)$ has relative $\mathcal T$, then $(\Gamma_2, A_2)$ has relative $\mathcal T$.
\end{enumerate}
\end{lemma}
\begin{proof} (i) Let $b: \Gamma_1 \to E$ be an $\mathcal E$-cocycle/quasi-cocycle/wq-cocycle. Since $\Gamma_1$ has finite index in $\Gamma_2$, it is a uniform lattice in $\Gamma_2$, and the invariant measure on $X = \Gamma_2/\Gamma_1$ is the normalized counting measure. Let $\widehat{E} = \ell^p(X; E)$ and denote by $\widehat{b}: \Gamma_2 \to \widehat{E}$ the induced cocycle/quasi-cocycle/wq-cocycle. We will assume that the section used to define the induction satisfies $s(\Gamma_1) = e$.
As in the proof of Proposition \ref{InductionWellDefined} we have for all $g \in \Gamma_2$,
\[\|\widehat{b}(g)\|^p \quad = \quad \sum_{\lambda \in \Gamma_1} \|b(\lambda)\|^p\; \nu(X(g, \lambda)),
\]
where $X(g, \lambda) = \{x \in X \mid \beta(g^{-1}, x)^{-1} = \lambda\}$. Since $\widehat{b}$ is bounded on $A_2 \supset A_1$, we deduce that there exists a constant $C$ such that for all $\lambda \in A_1$
\[
\|b(\lambda)\|^p \; \nu(X(\lambda, \lambda)) \quad \leq \quad \|\widehat{b}(\lambda)\|^p \quad \leq \quad C.
\]
Note that $\beta(\lambda^{-1}, \Gamma_1)^{-1} = s(\Gamma_1)^{-1}\lambda s(\lambda^{-1}\Gamma_1) = \lambda$, hence $X(\lambda, \lambda) \neq \emptyset$, and thus $\nu(X(\lambda, \lambda))  \geq  [\Gamma_2: \Gamma_1]^{-1}$. This shows that $b$ is bounded on $A_1$ and finishes the proof.

(ii) Let $b$ be an $\mathcal E$-cocycle/quasi-cocycle/wq-cocycle on $\Gamma_2$. Then $b|_{A_1}$ is bounded by assumption, and since $A_1$ is left-syndetic in $A_2$, it follows from the  wq-cocycle property that $b$ is bounded on $A_2$.
\end{proof}
Note that while the proof of (ii) also works in the topological setting, discreteness was essential in the proof of (i). The results of this subsection specialize to approximate groups in the obvious way. Note that if $(\Lambda, \Lambda^\infty)$ is an approximate group, then so is $(\Lambda^k, \Lambda^\infty)$ for every $k \in \mathbb N$, and we have:
\begin{corollary} Let $(\Lambda, \Lambda^\infty)$ be an approximate group and $\mathcal T$ be a Kazhdan type property with respect to a class $\mathcal E$ as in Lemma \ref{TUnderFiniteIndex}. Then $(\Lambda, \Lambda^\infty)$ has $\mathcal T$ if and only if  $(\Lambda^k, \Lambda^\infty)$ has $\mathcal T$ for every $k \in \mathbb N$.\qed
\end{corollary}

\subsection{The Haagerup Property and its relatives}
The Haagerup Property (also known as Gromov's a-T-menability) is a strong negation of Property (T) in the sense that every lcsc group which enjoys both Property (T) and the Haagerup Property is compact. In analogy with Property (T) (or rather Property (FH)), we can define a number of variants of this property as follows.

Let $G$ be a lcsc group with left-admissible metric $d$, and let $A$ be a subset. We say that an affine action $\rho=(\pi, b)$ of $G$ on a Banach space $E$ is \emph{metrically $A$-proper} if for every $C > 0$ and some (hence any) $v \in E$ the set $\{g \in A \mid \|\rho(g).v\|\leq C \}$ is relatively compact. Choosing $v = 0$ this is equivalent to $b$ being an \emph{$A$-proper cocycle}, i.e., pre-compactness of the sets $\{g \in A\mid \|b(g)\|\leq C\}$ for all $C > 0$. Similarly, a (weak) quasi-cocycle $b$ on $G$ is called $A$-proper if the corresponding sets are pre-compact.
\begin{definition} Let $G$ be a lcsc group and let $A$ be subset and let $\mathcal E$ be a class of Banach spaces.
\begin{enumerate}[(i)]
\item $G$ is a-F$\mathcal E$-menable relative to $A$ if there exists an $L^p$-cocycle on $G$ which is $A$-proper.
\item  $G$ is a-FF$\mathcal E$-menable relative to $A$ if there exists a weak $L^p$-cocycle on $G$ which is $A$-proper.
\item $G$ is a-FFF$\mathcal E$-menable relative to $A$ if there exists a weak $L^p$-quasi-cocycle on $G$ which is $A$-proper.
\end{enumerate}
If $A = G$ then we simply say that $G$ is a-F$\mathcal E$-menable, a-FF$\mathcal E$-menable or a-FFF$\mathcal E$-menable respectively. 
In particular, this defines absolute and relative versions of a-F$L^p$-menability, a-FF$L^p$-menability and a-FFF$L^p$-menability 
\end{definition}
\begin{remark}
\begin{enumerate}[(i)]
\item In the sequel we refer to a-F$\mathcal E$-menability, a-FF$\mathcal E$-menability or a-FFF$\mathcal E$-menability as \emph{Haagerup-type properties} with respect to $\mathcal E$.
\item Every Kazhdan-type property $\mathcal T$ has a dual Haagerup-type property $\mathcal H$ such that if $G$ enjoys both $\mathcal T$ and $\mathcal H$ relative to $A$, then $A$ is pre-compact. 
\item if $d$ denotes a left-admissible metric on $G$, then $A$-properness of a wq-cocycle $b$ is equivalent to the existence of a proper function $\rho: (0, \infty) \to (0, \infty)$ such that for all $x \in A$,
\[
d(x, e) \geq t \quad \Rightarrow \quad \|b(x)\| \geq \rho(t).
\]
Any such function $\rho$ is then called a \emph{control function} for $b$ on $A$.
\end{enumerate}
\end{remark}
\begin{definition} Let $\mathcal H$ be a Haagerup-type property. An approximate group $(\Lambda, \Lambda^\infty)$ is said to have $\mathcal H$ if $\Lambda^\infty$ has $\mathcal H$ relative to $\Lambda$.
\end{definition}

\section{Analytic properties of approximate lattices via induction}

Throughout this section, $\mathcal E$ denotes an $L^p$-closed class of uniformly convex separable Banach spaces for some $1 < p < \infty$. The example we have in mind is the the class of $L^p$-spaces, but no specific properties of $\mathcal E$ will be used.

\subsection{Preservation of Haagerup type properties under induction}\label{SecHaagerup}
In this section we relate analytic properties of strong uniform approximate lattices to analytic properties of the ambient lcsc group using cocycle induction. It turns out that this is rather straight-forward for the Haagerup type properties, but much harder for the Kazhdan type properties, hence we start with the former.

Throughout this subsection let $\Lambda$ be a strong uniform approximate lattice in a lcsc group $G$. We fix a left-admissible metric $d$ on $G$. Given $g \in G$ we abbreviate $\|g\| := d(g, e)$, and we denote by $R > r > 0$ the Delone-parameters of $\Lambda$ with respect to $d$. We then choose a Borel section $s: X_\Lambda \to G$ which takes values in $B_R(e)$ and denote by $\beta = \beta_s$ the associated cocycle. 

\begin{proposition}\label{HaagerupInduction}\label{wHaagerupInduction} Let $(\pi, b)$ be a wq-$\mathcal E$-pair with underlying Banach space $E$, let $\nu$ be a $G$-invariant probability measure $\nu$ on the hull $X_\Lambda$ and let $(\widehat{\pi}, \widehat{b}) := \nu{\rm -Ind}_{(\Lambda, \Lambda^\infty)}^G (\pi, b)$. If $b$ is $\Lambda$-proper, then $\widehat{b}$ is $G$-proper.
\end{proposition}
For the proof we will make use of the following trivial but important observation; recall that $d$ denotes a left-admissible metric on $G$.
\begin{lemma}\label{NoMilnorSchwarz} Let $d'$ be a left-admissible metric on $\Lambda^\infty$ and let $k \in \mathbb N$ and $A \subset \Lambda^k$. Then we have equivalences
\[
A \text{ is bounded w.r.t. }d \;\Leftrightarrow\; A \text{ is bounded w.r.t. }d'  \;\Leftrightarrow\; |A| < \infty. 
\]
\end{lemma}
\begin{proof} Since $\Lambda^\infty$ is discrete, $d'$-bounded sets are exactly the finite sets. If $\Lambda \subset G$ is an approximate lattice, then so is $\Lambda^k$ for every $k \in \mathbb N$. In particular, $\Lambda^k$ is uniformly discrete, hence intersects every $d$-bounded subset of $G$ in a finite set.
\end{proof}

\begin{proof}[Proof of Proposition \ref{wHaagerupInduction}]
Since the cocycle $b$ is $\Lambda$-proper and $\Lambda$ is syndetic in $\Lambda^2$, it is also $\Lambda^2$-proper. It thus follows from Lemma \ref{NoMilnorSchwarz} applied with $k = 2$ that there exists a proper function $\rho: (0, \infty) \to (0, \infty)$ such that for all $\lambda \in \Lambda^2$,
\begin{equation}\label{ProperHaagerupControl}
\|\lambda\| = d(\lambda, e) \geq t \quad \Rightarrow \quad \|b(\lambda)\| \geq \rho(t).
\end{equation}
Now consider the sets $X_\Lambda(g, \lambda)$ from \eqref{Xglambda}. It follows from \eqref{ProperHaagerupControl} and Lemma \ref{HaagerupTriangleInequality} that for all $g \in G$ and $\lambda \in \Lambda^2$ we have the implication
\begin{equation}
X_\Lambda(g, \lambda) \neq \emptyset \quad \Rightarrow \quad \|\lambda\| \geq \|g\|-4R \quad \Rightarrow \quad \|b(\lambda)\|\; \geq \; \rho(\|g\|-4R).
\end{equation}
Also observe that for fixed $g \in G$ the hull $X_\Lambda$ decomposes as the disjoint union
\[
X_\Lambda = \bigsqcup_{\lambda \in \Lambda^2} X_\Lambda(g, \lambda).
\]
We conclude that 
for all $g \in G$ with $\|g\| \geq 4 R$ we have
\begin{eqnarray*} \|\widehat{b}(g)\|^p &=& \int_{X_{\Lambda}} \|\widehat{b}(g)(P)\|^p \, d\nu(P) \quad = \quad  \int_{X_{\Lambda}} \|b(\beta(g^{-1}, P)^{-1})\|^p \,d\nu(P)\\
&=& \sum_{\lambda \in \Lambda^2} \|b(\lambda)\|^p \nu(X_\Lambda(g, \lambda)) \quad \geq \quad \sum_{\lambda \in \Lambda^2}\rho(\|g\|-2R)^p\nu(X_\Lambda(g, \lambda))\\
&=&\rho(\|g\|-4R)^p \cdot \sum_{\lambda \in \Lambda^2}\nu(X_\Lambda(g, \lambda)) \quad = \quad \rho(\|g\|-4R)^p,
\end{eqnarray*}
whence we obtain
\[
 \|\widehat{b}(g)\| \geq \rho(\|g\|-4R),
\]
showing that $\widehat{b}$ is proper.
\end{proof}

\begin{corollary}\label{CorHaagerupStrong} Let $\Lambda$ be a strong uniform approximate lattice in a lcsc group $G$ and let $\mathcal H$ be a Haagerup-type property with respect to $\mathcal E$.
Then $(\Lambda, \Lambda^\infty)$ has $\mathcal H$ if and only if $G$ has $\mathcal H$.
\end{corollary}
\begin{proof} Assume that $(\Lambda, \Lambda^\infty)$ is a-$F\mathcal E$-menable. Then there exists a $\Lambda$-proper affine $\mathcal E$-action of $\Lambda^\infty$, and by Proposition \ref{HaagerupInduction} we thus obtain a $G$-proper affine action of $G$ by induction, showing that $G$ is a-$F\mathcal E$-menable. Conversely, it follows from Lemma \ref{NoMilnorSchwarz} (with $k=1$) that if $G$ is a-$F\mathcal E$-menable, then any $G$-proper affine $\mathcal E$-action of $G$ restricts to a $\Lambda$-proper affine $\mathcal E$-action of $\Lambda^\infty$. The argument for the other Haagerup-type properties is the same.
\end{proof}
For instance, the following application of the corollary shows that (as in the group case) neither the a-$L^p$-menability for some $p>2$ nor the a-FFF$H$-menability imply a-F$H$-menability for approximate groups. 
\begin{example}\label{ExHaagerup} Assume that $\Lambda$ is a strong uniform approximate lattice in a rank one Lie group $G$. If $G$ is locally isomorpic to ${\rm SO}(1, n)$ or ${\rm SU}(1,n)$ for some $n \geq 2$, then $G$ and hence $(\Lambda, \Lambda^\infty)$ is a-F$H$-menable \cite{CCJVbook}. If $G$ is of quaternion or octonion type, then $G$ has Property (T) \cite{BHVbook}, hence is not a-F$H$-menable. It is however a-F$L^p$-menable for sufficiently large $p$ (e.g. for $p > 4n+2$ in the case of ${\rm Sp}(n,1)$, see \cite{CTV}). Moreover, as remarked by Ozawa \cite[Remark on p.2]{Oz11}, every countable hyperbolic group is a-FFF$H$-menable. This applies in particular to uniform lattices in $G$ (which exist by Borel--Harish-Chandra). Using both directions of the corollary, we deduce firstly that $G$ is a-FFF$H$-menable, and secondly that every strong uniform approximate lattice in a rank one group is a-FFF$H$-menable.
\end{example}
Corollary \ref{CorHaagerupStrong} is a special case of Theorem \ref{ThmHaagerupIntro} in that we have to assume here that the uniform approximate lattice under consideration is strong. In fact, Theorem \ref{ThmHaagerupIntro} can be deduced from Corollary \ref{CorHaagerupStrong} together with the following simple observation:
\begin{proposition} Let $G$ be a lcsc group and let $\mathcal H$ be a Haagerup property with respect to $\mathcal E$. If $\Lambda \subset \Sigma \subset G$ are uniform approximate lattices and $(\Sigma, \Sigma^\infty)$ has $\mathcal H$, then $(\Lambda, \Lambda^\infty)$ has $\mathcal H$.
\end{proposition}
\begin{proof} We prove this for a-F$\mathcal E$-menability, the other properties can be proved similarly. By assumption there exists a $\Sigma$-proper affine $\mathcal E$-action $\rho_0$ of $\Sigma^\infty$. We can restrict this action to an affine $\mathcal E$-action $\rho$ of $\Lambda^\infty$. Since $\Lambda$ and $\Sigma$ are both relatively dense in $G$, $\Lambda$ is left-syndetic in $\Sigma$ by Remark \ref{Syndetic}. If thus follows that $\rho_0$ and hence $\rho$ is also $\Lambda$-proper. This finishes the proof.
\end{proof}

\subsection{Preservation of Kazhdan type properties, I: The trivial direction}
Concerning Kazhdan type properties of approximate lattices we have the following trivial implication:
\begin{proposition} Let $\mathcal T$ be a Kazhdan type property with respect to $\mathcal E$. If $\Lambda$ is a uniform strong approximate lattice in a lcsc group $G$ and $(\Lambda, \Lambda^\infty)$ has $\mathcal T$, then $G$ has $\mathcal T$.
\end{proposition}
\begin{proof} Assume that  $(\Lambda, \Lambda^\infty)$ has Property $(F\mathcal E)$ and let $\rho = (\pi, b)$ be an affine $\mathcal E$-action of $G$. Then $\rho|_{\Lambda^\infty}$ is an affine $\mathcal E$-action of $\Lambda^\infty$, hence $\Lambda$ has a bounded orbit. Since $\Lambda$ is relatively dense in $G$ and $\rho$ is continuous, this implies that $G$ has a bounded orbit. The proof for the other properties is similar.
\end{proof}
\begin{remark} The argument applies more generally to pairs $(\Lambda, \Lambda^\infty)$ where $\Lambda \subset G$ is any subsets which is bi-syndetic in the sense that there exist compact subsets $K, L$ of $G$ such that $G = K\Lambda L$, and $\Lambda^\infty$ denotes its enveloping group. This includes in particular the class of strong non-uniform approximate lattices as introduced in \cite{BH}.
\end{remark}
\begin{problem}\label{ProblemT}  Let $\mathcal T$ be a Kazhdan type property, let $G$ be a lcsc group with $\mathcal T$ and let $\Lambda \subset G$ be a uniform approximate lattice with enveloping group $\Lambda^\infty$. Is it true that $(\Lambda, \Lambda^\infty)$ has $\mathcal T$?
\end{problem}
\begin{remark}
One might expect naively, that at least for strong uniform approximate lattices one can obtain a positive answer to this question by a similar induction argument as in the Haagerup case. However, any attempt in this direction runs into the following problem: Assume e.g.\ that $G$ has (F$L^p$) and let $(\pi, b)$ be an affine $L^p$-action of $\Lambda^\infty$, where $\Lambda \subset G$ is a strong uniform approximate lattice. Choose a $G$-invariant probability measure $\nu$ on $X_\Lambda$ and let $\widehat{b}$ denote the $\nu$-induced $L^p$-cocycle on $G$. The assumption then implies that $\widehat{b}$ is bounded, hence if we define sets $X_\Lambda(g, \lambda)$ as in \eqref{Xglambda}, then as in the proof of Proposition \ref{InductionWellDefined} we obtain for all $g \in G$,
\[
 \sum_{\lambda \in \Lambda^2} \|b(\lambda)\|^p \nu(X_\Lambda(g, \lambda)) \quad = \quad  \|\widehat{b}(g)\|^p  \quad \leq \quad \|\widehat{b}\|_\infty^p.
\]
Thus for all $\lambda \in \Lambda$ and $g \in G$ we have
\[
\|b(\lambda)\| \quad \leq \quad \frac{\|\widehat{b}\|_\infty}{\nu(X_\Lambda(g, \lambda))}.
\]
To conclude, we would have to find for every $\lambda \in \Lambda$ (or at least for sufficiently many such $\lambda$) some $g(\lambda) \in G$ such that $\nu(X_\Lambda(g(\lambda), \lambda) > C$ for some uniform $C > 0$. Even in the case where $\Lambda = \Lambda^\infty$ is a group, this requires some non-trivial argument (see \cite{Oz11}), but it can be done (for $g(\lambda):= \lambda$). In the approximate group case, the cocycle $\beta$ used in the definition of the sets $X_\Lambda(g, \lambda)$ is more complicated, and while we are able to show that $X_\Lambda(\lambda, \lambda)$ is open and non-empty, we are not able to conclude that it has uniformly positive measure in general.
\end{remark}
\subsection{Preservation of Kazhdan type properties, II: The model set case}
In this subsection we provide a positive answer to Problem \ref{ProblemT} in the model set case. Throughout this subsection let $\mathcal T$ be a Kazhdan type property with respect to $\mathcal E$, let $(G, H, \Lambda)$ be a uniform cut-and-project scheme and let $\Lambda = \Lambda(G, H, \Gamma, W)$ be a symmetric regular uniform model set containing the identity with window $W$. We denote by $\Lambda^\infty$ the enveloping group of $\Lambda$. Recall that $(\Lambda, \Lambda^\infty) \cong (\Gamma_W, \Gamma)$, where $\Gamma_W = \Gamma \cap (G \times W)$. In particular, $\mathcal T$ for $(\Lambda, \Lambda^\infty)$ is equivalent to relative $\mathcal T$ of the pair $(\Gamma, \Gamma_W)$.
We are going to show:
\begin{theorem}\label{TInductionModel} 
Let $(\pi, b)$ be a wq-$\mathcal E$-pair for $\Gamma$ with underlying space $E$ and let $(\widehat{\pi}, \widehat{b}) = {\rm Ind}_\Gamma^{G \times H}(\pi, b)$. If $\widehat{b}$ is bounded on $G \times \{e\}$, then $b$ is bounded on $\Gamma_W$.
\end{theorem}
As an immediate consequence we obtain:
\begin{corollary}\label{TInductionModelCor} Let $\mathcal T$ be a Kazhdan type property. If $\Lambda$ is a symmetric uniform model set containing the identity in a lcsc group $G$ and $G$ has $\mathcal T$, then
$(\Lambda, \Lambda^\infty)$ has $\mathcal T$.
\end{corollary}
\begin{proof} If $\Lambda = \Lambda(G, H, \Gamma, W)$, then $(\Gamma_W, \Gamma)$ has $\mathcal T$ by Theorem \ref{TInductionModel}. Since $(\Lambda, \Lambda^\infty) \cong (\Gamma_W, \Gamma)$, the corollary follows.
\end{proof}
Note that we do not assume here that $\Lambda$ has large window, hence Proposition \ref{ModelIsoModel} may not apply. In the notation of that proposition this means that the two induction schemes
\[\nu{-\rm Ind}_{(\Lambda, \Lambda^\infty)}^G (\pi_0, b_0) \qand \left.\left({\rm Ind}_\Gamma^{G \times H} (\pi, b)\right)\right|_G\]
may be different from each other, in which case we will work with the latter scheme. This scheme is, however, special to the model sets case, and we do not currently know how to give an intrinsic proof of Corollary \ref{TInductionModelCor} based on induction over the hull of $\Lambda$ without reference to an ambient lattice. This is the reason why our approach is currently limited to the model set case (see however Subsection \ref{SecTMeyer} for a slight extension).

Before we prove the theorem, let us mention that Corollary \ref{TInductionModelCor} implies that an approximate group $(\Lambda, \Lambda^\infty)$ can have Property (F$L^p$) for every $p \in (1, \infty)$, even though $\Lambda^\infty$ fails to have Property (F$L^p$) for any $p \in (1, \infty)$:
\begin{example}\label{ExTnoT} 
\begin{enumerate}[(i)]
\item Let $n \geq 2$, $G := {\rm SU}(n,2)$ and $H := {\rm SU}(n+1,1)$. Since $G$ has rank $\geq 2$, it has Property (F$L^p$) for all $p \in (1, \infty)$ by \cite[Thm. B]{BFGM}. On the other hand, $H$ has the Haagerup Property (see e.g. \cite{CCJVbook}), hence fails to have Property (T). It thus follows from \cite[Thm. A]{BFGM} that $H$ does not have Property (T$_{L^p}$) for any \cite[Thm. B]{BFGM}. It then follows from \cite[Thm. 1.3]{BFGM} that $H$, and hence $G\times H$, does not have Property (F$L^p$) for any $p \in (1, \infty)$. Now let $\Gamma < G \times H$ be an irreducible uniform lattice, let $W$ be a $\Gamma$-regular window which is symmetric and contains the identity, and let $\Lambda := \Lambda(G, H, \Gamma, W)$ be the associated regular uniform model set. It then follows from Corollary \ref{TInductionModelCor} that $(\Lambda, \Lambda^\infty)$ has Property (F$L^p$) for every $p \in (1, \infty)$. On the other hand, $\Lambda^\infty$ is isomorphic to the lattice $\Gamma$ in $G \times H$, hence fails to have Property (F$L^p$) for any $p \in (1, \infty)$. (Note that the case $p=2$ can also be deduced from the proof of \cite[Cor. 1.3]{CI}.)
\item Note that if $G$ and $\Lambda$ are as in (i), then $G$ and hence $\Lambda$ also have Property (FF$H$) by \cite{BM1}, while $\Lambda^\infty$ does not (since (FF$H$) would imply (F$H$)).
\end{enumerate} 
\end{example}
On a more anecdotal level we also mention:
\begin{example}\label{ExSp} The group $G := {\rm Sp}(n,1)$ has Property (FH) (see e.g. \cite{BHVbook}), but as we have seen in Example \ref{ExHaagerup} it is also
a-$FL^p$-menable for $p > 4n+2$ and a-FFF$H$-menable. The same then holds for any symmetric regular model set in $G$ containing the identity. This provides examples of approximate groups which have the above three properties simultaneously.
\end{example}

The remainder of this subsection is devoted to the proof of Theorem \ref{TInductionModel}. The main ingredient in the proof is the following theorem established by Narutaka Ozawa in his work on Property (TTT):
\begin{theorem}[{Ozawa, \cite[Thm. C]{Oz11}}]\label{ThmOzawa} Let $G$ be a lcsc group acting measure-preservingly on a standard probability space $(Y, m)$. Let $\ell: Y \times G \to \R_{\geq 0}$ be a measurable function such that for almost all $(x,g_1, g_2) \in Y \times G \times G$,
\begin{equation}\label{OzawaFirst}
\ell(x, g_1g_2) \leq \ell(x,g_1) + \ell(g_1^{-1}x, g_2).
\end{equation}
Assume that there exists $D > 0$ such that for almost all $g \in G$
\begin{equation}\label{OzawaSecond}
\int_Y \ell(x,g) dm(x) < D < \infty.
\end{equation}
Then there exists $\phi \in L^1(Y, m)$ with $\phi \geq 0$ such that for almost all $(g,x) \in G \times Y$,
\begin{equation}
\ell(x,g) \;\leq\; \phi(x) + \phi(g^{-1}x). 
\end{equation}
\qed
\end{theorem}
We now return to the situation of Theorem \ref{TInductionModel}. Concerning the cut-and-project scheme $(G, H, \Gamma)$ we use freely the notations introduced in Remark \ref{ModelSetConventions}. In particular we abbreviate $Y := (G \times H)/\Gamma$ and choose a bounded Borel section $s: Y \to G \times H$ as in the remark. We fix the section $s$ for the rest of this subsection and denote by $\beta = \beta_s$ the associated cocycle.

Given $\gamma \in \Gamma_G$ we denote by $\gamma^*$ the unique element of $\Gamma_H$ such that $(\gamma, \gamma^*) \in \Gamma$. Given $x, y\in Y$ and $(\gamma, \gamma^*) \in \Gamma$ we define elements $g^\gamma_{x,y} \in G$ and $h^\gamma_{x,y} \in H$ by
\begin{equation}\label{DefModelTgh}
s(x)(\gamma, \gamma^*)s(y)^{-1} = (g^\gamma_{x,y}, h^\gamma_{x,y})
\end{equation}
We will establish below the following useful algebraic identities:
\begin{lemma}\label{ModelTgh} 
\begin{enumerate}[(i)]
\item For all $(\gamma, \gamma^*) \in \Gamma$ and $x,y \in Y$we have
\[
(\gamma, \gamma^*) = \beta((s(x)(\gamma, \gamma^*)s(y)^{-1})^{-1}, x)^{-1}.
\]
\item For all $(\gamma, \gamma^*) \in \Gamma$ and all $x,y \in Y$ we have
\[
(g_{x,y}^\gamma, e)^{-1}.x = (s_G(y), s_H(x)\gamma^*).\Gamma = (e, h_{x,y}^\gamma).y.
\]
\end{enumerate}
\end{lemma}
We also need the following lemma, which is essentially a combination of Fubini's theorem and \v{C}eby\v{s}ev's inequality. Here our special choice of section will play an important role.
\begin{lemma}\label{LemmaFubini} 
Let $\Omega \subset G \times Y$ be a conull subset and let $\phi \in L^1(Y, m)$ with $\phi \geq 0$. Then there exists a constant $C > 0$ such that for all $(\gamma, \gamma^*) \in \Gamma \cap (G \times W)$ there exist $x,y \in Y$ such that
\[
(g^\gamma_{x,y}, x) \in \Omega \quad \text{and} \quad \phi(x) \leq C \quad \text{and} \quad \phi((g^\gamma_{x,y}, e)^{-1}x) \leq C.
\]
\end{lemma}
Before we prove the two lemmas, let us explain how they imply the theorem:
\begin{proof}[Proof of Theorem \ref{TInductionModel} modulo Lemma \ref{ModelTgh} and Lemma \ref{LemmaFubini}] 
Our arguments roughly decompose as follows: We first show that there exists a constant $C'$ such that
\[
\|b(\gamma,\gamma^*)\| \leq \|b((g^{\gamma}_{x,y},e)^{-1},x)\| + C', \quad \textrm{for all $x, y \in Y$ and $(\gamma, \gamma^*) \in \Gamma_W$},
\]
where $g_{x,y}^\gamma$ is defined in \eqref{DefModelTgh}.
Once this has been established, we verify that Ozawa's Theorem can be applied to the function $\ell(x, g) := \|b(\beta((g,e)^{-1},x)^{-1})\| + C'$, which in turn implies that there exist $\phi \in L^1(Y,m)$ and a conull subset $\Omega \subset G \times Y$ such that
\[
\|b(\beta((g,e)^{-1},x)^{-1})\| \leq \phi(x) + \phi((g,e)^{-1}x), \quad \textrm{for all $(g,x) \in \Omega$}.
\]
Now, by Lemma \ref{LemmaFubini}, there exists a constant $C$ such that for all $(\gamma,\gamma^*) \in \Gamma_W$, there are 
$x, y \in Y$ which satisfy
\[
(g^\gamma_{x,y}, x) \in \Omega, \quad \phi(x) \leq C \quad \text{and} \quad \phi((g^\gamma_{x,y}, e)^{-1}x) \leq C,
\]
whence, $\|b(\beta((g^{\gamma}_{x,y},e)^{-1},x))\| \leq 2C$, and thus $\|b(\gamma,\gamma^*)\| \leq 2C + C'$. Since the constants 
$C$ and $C'$ are independent of $x$ and $y$, we conclude that $b |_{\Gamma_W}$ is bounded, which finishes the proof. \\

Let us now turn to the details. By Lemma \ref{ModelTgh}  we have for all $x,y \in Y$ and $(\gamma, \gamma^*) \in \Gamma$,
\begin{eqnarray*}
(\gamma, \gamma^*) &=&  \beta((s(x)(\gamma, \gamma^*)s(y)^{-1})^{-1}, x)^{-1} \quad = \quad \beta((g^\gamma_{x,y}, h^\gamma_{x,y})^{-1}, x)^{-1}\\
&=& \beta\left((e, h^\gamma_{x,y})^{-1}(g^\gamma_{x,y},e)^{-1}, x\right)^{-1}\\
&=& \beta((g^\gamma_{x,y},e)^{-1}, x)^{-1} \beta((e,h^{\gamma}_{x,y})^{-1}, (g^\gamma_{x,y},e)^{-1}x)^{-1}\\
&=&  \beta((g^\gamma_{x,y},e)^{-1}, x)^{-1} \beta((e,h^{\gamma}_{x,y})^{-1}, (e,h^{\gamma}_{x,y}).y) ^{-1}\\
&=&  \beta((g^\gamma_{x,y},e)^{-1}, x)^{-1} \beta((e,h^{\gamma}_{x,y})^{-1}, y),
\end{eqnarray*}
where the last equality follows from \eqref{betaInverse}. In particular, we deduce from the wq-Property and the fact that $\pi$ takes values in isometries that
\[
\|b(\gamma, \gamma^*)\| \leq  \|b(\beta((g^\gamma_{x,y},e)^{-1}, x)^{-1})\| + \|b( \beta((e,h^{\gamma}_{x,y})^{-1}, y))\| + D(b),
\]
where $D(b)$ is defined in \eqref{def_Db}. Now assume that $(\gamma, \gamma^*) \in \Gamma_W := \Gamma \cap (G \times W)$. Then
\[
h^\gamma_{x,y} = s_H(x)\gamma^*s_H(y) \in \pi_H(\mathcal F)W\pi_H(\mathcal F),
\]
which is a compact subset of $H$. It thus follows from the last assertion of Lemma \ref{LemmabetaCocycle} that there exists a constant $C_0 > 0$ such that for all $h \in  \pi_H(\mathcal F)W\pi_H(\mathcal F)$ and $y \in Y$
\[
\|b(\beta((e,h), y))\| \leq C_0,
\]
and hence
\begin{equation}\label{PreOzawa}
\|b(\gamma, \gamma^*)\| \; \leq \; \|b(\beta((g^\gamma_{x,y},e)^{-1}, x)^{-1})\| + C_0 +D(b)
\end{equation}
for all $x, y \in Y$ and $(\gamma, \gamma^*) \in \Gamma_W$. To estimate the first term on the right-hand side of the inequality we are going to apply Ozawa's theorem to the function
\[
 \ell: Y \times G \to \R_{\geq 0}, \quad \ell(x, g) := \|b(\beta((g,e)^{-1},x)^{-1})\| + D(b).
\]
We need to check that this function satisfies Conditions \eqref{OzawaFirst} and \eqref{OzawaSecond} from Theorem \ref{ThmOzawa}. Set
\[
C_1 := \sup_{g \in G}\|\widehat{b}(g,e)\|_p < \infty.
\]
Since for all $g \in G$ and $x \in Y$ we have $\widehat{b}(g,e)(x) = b(\beta((g,e)^{-1}, x)^{-1})$, we deduce 
that
\[
\int_Y \|b(\beta((g,e)^{-1}, x)^{-1})\| dm(x)  \quad \leq \quad
\left(\int_Y \|b(\beta((g,e)^{-1}, x)^{-1})\|^p dm(x) \right)^{1/p}  \quad \leq \quad C_1.
\]
This implies that
\begin{equation*}\label{OzawaApplies}
\int_Y \ell(x,g) dm(x) \quad \leq \quad C_1 + D(b) \quad < \quad \infty,
\end{equation*}
hence \eqref{OzawaSecond} holds with $D := C_1 + D(b)$. We also observe that
\begin{eqnarray*}
\ell(x,g_1g_2) &=&   \|b(\beta((g_1g_2,e)^{-1},x))^{-1}\| +D(b)\\
&=&\|b(\beta((g_1,e)^{-1},x)^{-1}\beta((g_2, e)^{-1}, (g_1, e)^{-1}.x^{-1})^{-1}\| +D(b)\\
&\leq& \|b(\beta((g_1,e)^{-1},x)^{-1})\| + \|b(\beta((g_2, e)^{-1}, (g_1, e)^{-1}x^{-1})^{-1})\| + 2D(b)\\
&\leq& \ell(x, g_1) + \ell(g_1^{-1}x, g_2),
\end{eqnarray*}
hence \eqref{OzawaFirst} holds, and Theorem \ref{ThmOzawa} applies indeed to the function $\ell$. We thus deduce that there exists a conull set $\Omega \subset G \times Y$ and a function $\phi \in L^1(Y, m)$ 
such that for all $(g, x) \in \Omega$,
\begin{equation}\label{DefOmega0}
\|b(\beta((g,e)^{-1},x)^{-1})\| \quad \leq \quad \phi(x) + \phi((g,e)^{-1}x).
\end{equation}
Now fix $(\gamma, \gamma^*) \in \Gamma_W$ and let $C$ be the constant from Lemma \ref{LemmaFubini}. By the same lemma, there exist $x,y \in Y$ such that $(x,g^\gamma_{x,y}) \in \Omega$, $\phi(x) \leq C$ and $\phi(g^\gamma_{xy},e)^{-1}x) \leq C$. Now by \eqref{PreOzawa} and \eqref{DefOmega0} this implies that 
\[
\|b(\gamma, \gamma^*)\| \; \leq\;  \|b(\beta(x, (g^\gamma_{xy}, e)))\| + C_0 \; \leq \; \phi(x) + \phi((g^\gamma_{xy}, e)^{-1}x) +C_0 \; \leq \; 2C + C_0.
\]
Since $(\gamma, \gamma^*) \in \Gamma_W$ was chosen arbitrarily, this shows that $b$ is bounded on $\Gamma_W$ and thereby finishes the proof.
\end{proof}
The rest of this subsection is devoted to the proof of the lemmas used in the proof of this theorem.
\begin{proof}[Proof of Lemma \ref{ModelTgh}] (i) Since $\beta(g^{-1},x) ^{-1} = s(x)^{-1}gs(g^{-1}x)$ we have
\begin{eqnarray*}
\beta((s(x)(\gamma, \gamma^*)s(y)^{-1})^{-1}, x)^{-1} &=& s(x)^{-1}(s(x)(\gamma, \gamma^*)s(y)^{-1})s(s(y)(\gamma, \gamma^*)s(x)^{-1}x)\\
&=& (\gamma, \gamma^*) s(y)^{-1}s(s(y) (\gamma, \gamma^*)\Gamma)\\
&=&  (\gamma, \gamma^*) s(y)^{-1}s(y) \quad \quad = \quad (\gamma, \gamma^*).
\end{eqnarray*}
(ii)  We have
\begin{eqnarray*}
(s_G(x)\gamma s_G(y)^{-1}, s_H(x)\gamma^* s_H(y)^{-1}) = s(x)(\gamma, \gamma^*)s(y)^{-1} = (g^\gamma_{x,y}, h^\gamma_{x,y}),
\end{eqnarray*}
and hence
\begin{eqnarray*}
(g_{x,y}^\gamma, e)^{-1}.x &=& (s_G(y)\gamma^{-1} s_G(x)^{-1}, e).(s_G(x), s_H(x))(\gamma, \gamma^*)\Gamma\\
&=& (s_G(y), s_H(x)\gamma^*)\Gamma\\
&=& (e, s_H(x)\gamma^* s_H(y)^{-1})(s_G(y), s_H(y))\Gamma\\
&=& (e, h^\gamma_{x,y}).y,
\end{eqnarray*}
which establishes $(g_{x,y}^\gamma, e)^{-1}.x = (s_G(y), s_H(x)\gamma^*).\Gamma = (e, h_{x,y}^\gamma).y$.
\end{proof}
\begin{proof}[Proof of Lemma \ref{LemmaFubini}] Throughout the proof we fix $(\gamma, \gamma^*) \in \Gamma \cap (G \times W)$. For every $C >0$ we now define sets
\[
\Omega_1^\gamma := \{(x,y) \in Y \times Y \mid  (g^\gamma_{x,y}, x) \in \Omega\} \quad \text{and} \quad \Omega_2^C := \{(x,y) \in Y \times Y \mid \phi(x) \leq C\}
\]
and
\[
\Omega_3^{\gamma, C} := \{(x,y) \in Y \times Y \mid 
\phi((g^\gamma_{x,y}, e)^{-1}x) \leq C\}.
\]
It suffices to show that for some $C>0$ independent of $\gamma$
the intersection
\[
\Omega_{\gamma, C} := \Omega_1^\gamma \cap \Omega_2^C \cap \Omega_3^{\gamma, C},
\]
has positive measure (hence is non-empty).

As far as $\Omega_1^\gamma$ is concerned, since $\Omega$ is a conull subset of $G \times Y$, Fubini's theorem implies that there exists a conull subset $Y_0 \subset Y$ such that for all $x \in Y_0$ the set
\[
\Omega_x := \{g \in G\mid (g, x) \in \Omega\}
\]
is a conull set in $G$. Now fix $x \in Y_0$ and let $y \in Y$. By definition we have the equivalences
\[
(x,y) \in \Omega_1^\gamma \quad \Leftrightarrow \quad s_G(x)\gamma s_G(y)^{-1} \in \Omega_x \quad \Leftrightarrow \quad s_G(y) \in \Omega_x^{-1}s_G(x)\gamma.
\]
Now for every fixed $x \in Y_0$ the set $\Omega_x^{-1}s_G(x)\gamma$ is a conull set in $G$. This shows that
\begin{equation}\label{Omega1Ozawa}
(m_Y \otimes m_Y)(\Omega_1^\gamma) = 1.
\end{equation}
As far as $\Omega_2^C$ is concerned, \v{C}eby\v{s}ev's inequality yields
\[
(m_Y \otimes m_Y)(\Omega_2^C)  \;=\; 1-m_Y(\{x \in Y \mid  \phi(x) > C\}) \;\geq\; 1-\frac{1}{C}\|\phi\|_1,
\]
hence if $\varepsilon > 0$ and $C >\|\phi_1\|/\varepsilon$, then 
\begin{equation}\label{Omega2Ozawa}
(m \otimes m)(\Omega_2^C) > 1-\varepsilon.
\end{equation}

To estimate the measures of the sets $\Omega_3^{\gamma, C}$ we define a compact subset $K \subset G \times H$ by $K := \pi_G(\mathcal F) \times \pi_H(\mathcal F) W$ and define 
\[
\widetilde{\phi}: G \times H \to \R_{\geq 0}, \quad \widetilde{\phi}(g,h) := \chi_K(g,h) \cdot \phi((g,h)\Gamma).
\]
Since $K \subset G \times H$ is compact, it can be covered by finitely many translates of $\mathcal F$, and since $\phi \in L^1(Y, m)$ we deduce that $\widetilde{\phi} \in L^1(G \times H, m_G \otimes m_H)$. We also define a function
\[
h_\gamma: Y \times Y, \quad (x,y) \mapsto \widetilde{\phi}(s_G(y), s_H(x)\gamma^*) 
\]
Since by Lemma \ref{ModelTgh} we have $(g^\gamma_{x,y}, e)^{-1}x =  (s_G(y), s_H(x)\gamma^*).\Gamma$ and since for all $x,y \in Y$ we have $(s_G(y), s_H(x)\gamma^*) \in K$ we deduce that
\[
\Omega_3^{\gamma, C} = \{(x,y) \in Y \times Y \mid h_\gamma(x,y) \leq C\}.
\]
We observe that, in the notation of Lemma \ref{FMeasures},
\begin{eqnarray*}
\|h_\gamma\|_1 &=&\int_{Y \times Y} \widetilde{\phi}(s_G(y), s_H(x)\gamma^*) dm_Y(y) dm_Y(x)\\
&=& \int_{G} \int_{H} \widetilde{\phi}(g, h\gamma^*) d\mu_H(h) d\mu_G(g)\\
&=& \int_{\mathcal F_G} \int_{\mathcal F_H} \widetilde{\phi}(g, h\gamma^*)  \rho_G(g) \rho_H(h) dm_H(h) dm_G(g)\\
&\leq& \|\rho_G\|_\infty \|\rho_H\|_\infty \|\widetilde{\phi}\|_1 \quad=: \quad C_0,
\end{eqnarray*}
in particular the functions $h_\gamma$ are contained in $L^1(Y \times Y)$ and have uniformly bounded $L^1$-norm. Another application of \v{C}eby\v{s}ev's inequality thus yields
\[
(m_Y  \otimes m_Y)(\Omega_3^{\gamma, C}) \quad \geq \quad 1- \frac{C_0}{C},
\] 
which for $C > C_0/\varepsilon$ yields
\begin{equation}\label{Omega3Ozawa}
(m_Y  \otimes m_Y)(\Omega_3^{\gamma, C}) \quad > \quad 1 - \varepsilon.
\end{equation}
Combining \eqref{Omega1Ozawa}, \eqref{Omega2Ozawa} and \eqref{Omega3Ozawa} we deduce that for every $\epsilon > 0$ there exists $C > 0$ (independent of $\gamma$) such that
\[
(m_Y \otimes m_Y)(\Omega_{\gamma, C}) > 1 - 2\varepsilon,
\]
hence $\Omega_{\gamma, C} \neq \emptyset$.
\end{proof}

\subsection{Preservation of Kazhdan type properties, III: The Meyer case}\label{SecTMeyer}
In the last subsection we have established Part (1) of Theorem \ref{KazhdanType} in Corollary \ref{TInductionModelCor}. The remaining part of the theorem can be deduced from Part (1) using the results of the appendix:
\begin{proof}[Proof of Theorem \ref{KazhdanType}(2)] Let $G$ be a lcsc group satisfying a Kazhdan type Property $\mathcal T$ and let $\Lambda \subset G$ be a Meyer set, which is a uniform lattice and contained in a model set $\Sigma$ of almost connected type. By Part (1) of Theorem \ref{KazhdanType}, the approximate group $(\Sigma, \Sigma^\infty)$ has $\mathcal T$. It then follows from Corollary \ref{MeyerTrick} that also $(\Lambda, \Lambda^\infty)$ has $\mathcal T$.
\end{proof}
This finishes the proof of Theorem \ref{KazhdanType}

\newpage
\appendix

\section{Structure theory of Meyer sets}

Let $(G, H, \Gamma)$ be a uniform cut-and-project scheme and let $W \subset H$ be a subset with dense interior. Recall that the set
\[
\Lambda := \Lambda(G, H, \Gamma, W) := \pi_G(\Gamma \cap (G \times W))
\]
is called a \emph{uniform model} set with \emph{window} $W$. If $W$ is symmetric and contains the identity, then it is a uniform approximate lattice in $G$.
By Remark \ref{Syndetic} every symmetric subset $\Delta < \Lambda$ containing the identity, which is relatively dense in $G$ (equivalently, left-syndetic in $\Lambda$), is then also a uniform approximate lattice in $G$.
\begin{definition} A left-syndetic subset of a model set is called a \emph{Meyer set}.
\end{definition}

If $G$ is abelian, then every uniform approximate lattice in $G$ is a symmetric Meyer set \cite{Meyerbook, Moody}. As for general $G$, it is currently not known whether there exist any uniform approximate lattices which are not symmetric Meyer sets.

\begin{remark} In the definition of a model set, we can always assume that $W$ generates $H$, for otherwise we can replace $H$ by the group generated by $W$ without changing $\Lambda$. Similarly, in the definition of a Meyer set, we may always assume that the ambient model set is symmetric, regular and contains the identity, hence is a strong approximate lattice. Indeed, this can be achieved by simply enlarging the window. In the sequel, when constructing model sets/Meyer sets, we will always assume tacitly that $H$ and $W$ are chosen in this way.
\end{remark}
\begin{definition} Let $\Lambda = \Lambda(G, H, \Gamma, W)$ be a uniform model set and $\Delta \subset \Lambda$ be a Meyer set.
We say that $\Lambda$ and $\Delta$ are of \emph{connected Lie type} if $H$ is a connected Lie group. We say that $\Lambda$ and $\Delta$ are of \emph{almost connected type} if $H$ is almost connected, i.e., compact-by-connected.
\end{definition}
The difference between almost connected type and connected Lie type is rather small:
\begin{proposition} Let $G$ be a lcsc group. If $\Lambda$ is a uniform model set of almost connected type in $G$, then there exists a uniform model set $\Lambda'$ of connected Lie type in $G$ and a finite subset $F\subset \Lambda^{-1}\Lambda$ such that 
\[
\Lambda \subset \Lambda' \subset \Lambda F
\]
In particular, every Meyer set of almost connected type is of connected Lie type.
\end{proposition}
\begin{proof} Clearly the first statement implies the second.  As for the first statement, let $\Lambda = \Lambda(G, H, \Gamma, W_0)$ with $H$ almost connected. Then by \cite[p. 175]{MZ55} there exist a compact normal subgroup $V < H$ such that $L := H/V$ is a connected Lie group and $\{e\} \times V$ intersects $\Lambda$ trivially. Denote by $\pi_L: H \to L$ the canonical projection and set $W_L := \pi_L(W_0)$, $W := \pi_L^{-1}(W_0) = W_0V$ and 
\[\Gamma_L := \{(\gamma_1, \pi_L(\gamma_2)) \in G \times L \mid (\gamma_1, \gamma_2) \in \Gamma\} < G \times L
\]
Then $(G, L, \Gamma_L)$ is a cut-and-project scheme, $W$ is compact, and
\[
\Lambda(G, H, \Gamma, W) = \Lambda(G, L, \Gamma_L, W_L)
\]
is a uniform model set of connected Lie type. Since $\Lambda$ and $\Lambda(G, H, \Gamma, W)$ are model sets associated with the same cut-and-project scheme with windows $W_0 \subset W$, the model set $\Lambda$ is left-syndetic in this model set of connected Lie type by Remark \ref{Syndetic}.
\end{proof}
One reason for our interest in model sets of almost connected type is the following observation:
\begin{proposition}\label{MeyerFiniteIndex} Let $\Lambda$ be a regular symmetric uniform model set of almost connected type and let $\Delta \subset \Lambda$ be a symmetric Meyer set containing the identity. Then $\Delta^\infty$ is of finite index in $\Lambda^\infty$.
\end{proposition}
In connection with Kazhdan type properties we mention the following application, which is immediate from Proposition \ref{MeyerFiniteIndex} and Lemma \ref{TUnderFiniteIndex}:
\begin{corollary}\label{MeyerTrick}  Let $\mathcal T$ be a Kazhdan type property with respect to a class of uniformly convex separable Banach spaces which is $L^p$-closed for some $p\in (1, \infty)$. Let $\Lambda$ be a symmetric uniform regular model set of almost connected type and let $\Delta \subset \Lambda$ be a symmetric Meyer set containing the identity. If $(\Lambda, \Lambda^\infty)$ has $\mathcal T$, then $(\Delta, \Delta^\infty)$ has $\mathcal T$.\qed
\end{corollary}
For the proof of Proposition \ref{MeyerFiniteIndex} we need:
\begin{lemma}\label{ModelSet} Let $G$, $L$ be lcsc groups, $W_L \subset L$ be compact and let $\Theta \subset G \times L$ be a subset such that $\pi_L(\Theta)$ is dense in $L$. If 
\[\Sigma := \pi_G(\Theta \cap (G \times W_L))\]
is relatively dense in $G$, then $\Theta$ is relatively dense in $G \times L$.
\end{lemma}
\begin{proof} Since $\Sigma$ is relatively dense, we can choose a compact subset $K_1 \subset G$ such that $G = \Sigma K_1$. Let $K_2$ be a compact identity neighbourhood in $L$ and observe that since $\pi_L(\Theta)$ is dense in $L$ we have $L = \pi_2(\Theta)K_2$. We claim that $G \times L = \Theta(K_1 \times W_L^{-1}K_2)$. 

Indeed, let $(g,l) \in G \times L$; since $L = \pi_2(\Theta)K_2$ we can write $l = \theta_2k_2$ for some $\theta = (\theta_1, \theta_2) \in \Theta$ and $k_2 \in K_2$. Since $G = \Sigma K_1$ we then find $\sigma \in \Sigma $ and $k_1 \in K_1$ such that $\theta_1^{-1}g = \sigma k_1$. By definition of $\Sigma$ we can write $\sigma = \pi_G(\theta')$ for some $\theta' = (\sigma, \theta_2') \in \Theta$ with $\theta_2' \in W_L$. Then
\[
(g, l) = (\theta_1, \theta_2)(\theta_1^{-1}g, k_2) = \theta (\sigma, \theta_2') (k_1, (\theta_2')^{-1}k_2) = \theta\theta'  (k_1, (\theta_2')^{-1}k_2),
\]
and since $(g,l) \in G \times L$ was chosen arbitrarily we have $G \times L = \Theta(K_1 \times W_L^{-1}K_2)$ as claimed.
\end{proof}
\begin{proof}[Proof of of Proposition \ref{MeyerFiniteIndex}] Let $\Lambda = \Lambda(G, H, \Gamma, W)$ with $H$ almost connected and denote by $\tau: \Lambda^\infty = \Gamma_G \to H$ the $*$-map of the cut-and-project scheme $(G, H, \Gamma)$.  Define 
\[
\Theta \;:=\;  \Gamma \cap \pi_G^{-1}(\Delta^\infty) \;<\; \Gamma
\]
so that $\pi_G(\Theta) = \Delta^\infty$ and $\pi_H(\Theta) = \tau(\Delta^\infty)$.
By assumption there exists a finite subset $F \subset \Lambda^\infty$ such that $\Lambda \subset F\Delta$. Thus if we set $F^* := \tau(F)$, then
\[
W \cap \Gamma_H \;=\; \tau(\Lambda) \;\subset \;\tau(F\Delta) \;=\; F^* \tau(\Delta).
\]
Since $\Gamma_H$ is dense in $H$, the intersection $W \cap \Gamma_H$ is dense in $W^o$, hence in $\overline{W^o} = W$. We deduce that
\[
W \;=\; \overline{W \cap \Gamma_H} \;\subset\; F^* \overline{\tau(\Delta)}.
\]
Thus $F^*\overline{\tau(\Delta)}$ has non-empty interior, and since $F^*$ is finite, the Baire category theorem implies that $\overline{\tau(\Delta)}$ has non-empty interior. It follows that the subgroup 
\begin{equation}\label{DenseLinH}
L := \overline{ \pi_H(\Theta)} = \overline{\tau(\Delta^\infty)} < H
\end{equation}
has non-empty interior, hence is an open subgroup of $H$, and $W_L := W\cap L$ is a non-empty compact subset of $L$. By definition $\Theta$ is contained in $G \times L$, and by
\eqref{DenseLinH}, $\pi_L(\Theta) = \pi_H(\Theta)$ is dense in $L$. Moreover the set
\[
\Sigma :=  \pi_G(\Theta \cap (G \times W_L))
\]
is relatively dense in $G$, since it contains $\Delta$. Thus Lemma \ref{ModelSet} applies, and we deduce that $\Theta$ is relatively dense in $G \times L$. Now $H$ is almost connected and $L$ being open contains the identity component of $H$, hence $L$ is left-syndetic in $H$. It follows that $\Theta$ is also a relatively dense in $G\times H$.
Since $\Theta < \Gamma$ it is also uniformly discrete, hence a uniform lattice. Since $\Theta$ and $\Gamma$ are both uniform lattices, we deduce that $\Theta$ has finite index in $\Gamma$, and projecting to $G$ we see that $\Delta^\infty = \pi_G(\Theta)$ has finite index in $\Lambda^\infty = \pi_G(\Gamma)$.
\end{proof}
Finally we turn to the question how far an arbitrary Meyer set is from being of almost connected type, respectively connected Lie type. To state our result, we introduce the following terminology:

\begin{definition} Let $\Gamma$ be a group and $\Sigma \subset \Gamma$ be a symmetric subset. We say that an element $a \in \Gamma$ \emph{quasi-commensurates} $\Sigma$ if there exists a finite subset $F_a \subset \Gamma$ such that 
\[
a\Sigma \subset \Sigma F_a \quad \text{and} \quad \Sigma a \subset F_a \Sigma.
\] 
We say that a subset $A \subset \Gamma$ quasi-commensurates the set $\Sigma$ if every element of $A$ quasi-commensurates $\Sigma$. If $\Lambda \subset \Gamma$ is a subset we denote by ${\rm qComm}_{\Lambda}(\Sigma)$ the set of all elements of $\Lambda$ which commensurate $\Sigma$.
\end{definition}

B definition,  ${\rm qComm}_{\Lambda}(\Sigma)$ is the largest subset of $\Lambda$ which quasi-commensurates $\Sigma$. Note that we can enlarge approximate groups by finite subsets of their quasi-commensurator:
\begin{proposition}\label{AlmostConnectedvsConnectedLie} Let $G$ be a group and $\Sigma \subset G$ be an approximate subgroup. If $F \subset {\rm qComm}_G(\Sigma)$ is finite and contains the identity, then $F\Sigma\cup \Sigma F^{-1}$ is an approximate subgroup of $G$.
\end{proposition}
\begin{proof} By construction, $\Sigma_F$ is symmetric and contains the identity. Let $F_\Sigma$ be finite such that $\Sigma^2 \subset \Sigma F_\Sigma$. Given $x \in F$ chose $F_x$ finite such that $x\Sigma \subset \Sigma F_x$ and set $F_0 := \bigcup F_x$. Then
\[
\Sigma_F^2 \subset \Sigma\Sigma F_0^3 \subset \Sigma F_\Sigma F_0^3,
\]
which shows that $\Sigma_F$ is an approximate subgroup.
\end{proof}
\begin{definition} If $G$ is a group, $\Sigma \subset G$ is an approximate subgroup and $F \subset {\rm qComm}_G(\Sigma)$ is finite, then the approximate group $\Sigma_F := F\Sigma\cup \Sigma F^{-1}$ is called the \emph{enlargement} of $\Sigma$ by $F$.
\end{definition}
While it is not true in general that every Meyer set is contained in a model set of connected Lie type, we can show that every Meyer set is contained in the enlargement of a suitable model set of connected Lie type.
\begin{theorem}\label{ThmMeyerGeneral} Let $G$ be a lcsc group and let $\Lambda$ be a model set in $G$. Then there exists a model set $\Sigma$ of connected Lie type in $G$ and a finite subset $F \subset {\rm qComm}_{\Lambda^\infty}(\Sigma)$ such that
\[
\Lambda \subset F\Sigma \subset \Sigma_F.
\]
In particular, every Meyer set is contained in a finite union of model sets of connected Lie type.
\end{theorem}
\begin{proof} Let $(G, H, \Gamma)$ be a cut-and-project scheme with $*$-map $\tau: \Gamma_G \to H$ and let $\Lambda = \Lambda(G, H, \Gamma, W) = \tau^{-1}(W)$. Denote by $H_o$ the identity component of $H$ and by $\pi: H \to H/H_o$ be the canonical projection. Let $U<H/H_o$ be an arbitrary compact-open subgroup and set $L := \pi^{-1}(U)$ so that $L$ is an almost connected open subgroup of $H$. Since $L$ is almost connected, \cite[p. 175]{MZ55} implies that we can choose a compact normal subgroup $V$ of $L$ such that $\{e\} \times L$ intersects $\Gamma$ trivially and such that $M := L/V$ is a connected Lie group. We denote by $\pi_M: L \to M$ the canonical projection.

Since $\Gamma_H$ is dense in $H$, the image $\pi(\Gamma_H) = \pi(\tau(\Gamma_G))$ is dense in $H/H_o$ and thus $\pi(\tau(\Gamma_G))U = H/H_o$. Since $\pi(W)$ is compact we find a finite subset $F \subset \Gamma_G$ such that $\pi(W) \subset \pi(\tau(F))U$, and hence
\[
W \subset \tau(F) L.
\]
We now choose a compact subset $W_L \subset L$ such that $W_L \supset F^{-1}W \cap L \subset L$ and such that $W_M := \pi_M(W_L)$ satisfies $\pi_M^{-1}(W_M) = W_L$. Then
\[
W = \bigcup_{f \in F} (W \cap \tau(f)L) = \bigcup_{f \in F} \tau(f) (\tau(f)^{-1}W \cap L) \subset \bigcup_{f \in F} \tau(f)W_L = \tau(F)W_L,
\]
and hence if we set $\Sigma := \tau^{-1}(W_L) \subset \Gamma_G$, then
\[
\Lambda = \tau^{-1}(W) \subset \tau^{-1}(\tau(F)W_L) = F \tau^{-1}(W_L) = F\Sigma.
\]
Now define $\Gamma_L :=  \Gamma \cap (G \times L) < G \times L$ and $\Gamma_M := ({\rm id} \times \pi_M)(\Gamma_L) < G \times M$. Since $G\times L$ is open in $G\times H$, the group $\Gamma_L$ is a uniform lattice in $G \times L$. Since $\Gamma_L < \Gamma$ its projection to $G$ is injectice, and since $L$ is open in $H$, the projection of $\Gamma_L$ to $L$ is dense in $L$, hence $(G, L, \Gamma_L)$ is a cut-and-project scheme. From this it follows as in the proof of Proposition \ref{AlmostConnectedvsConnectedLie} that also $(G, M, \Gamma_M)$ is a cut-and-project scheme and since $W_L = \pi_M^{-1}(W_M)$ we have
\[
\Sigma = \pi_G((G \times W_L) \cap \Gamma) = \pi_G((G \times W_L) \cap \Gamma_L) =  \pi_G((G \times W_M) \cap \Gamma_M) = \Lambda(G, M, \Gamma_M, W_M).
\]
This shows that $\Sigma$ is a model set of connected type such that $\Lambda \subset F\Sigma$. Moreover, if $x \in F$, then the compact set $\tau(x)W_L$ can be covered by finitely many $\Gamma_H$-translates of $W_L$ (since $\Gamma_H$ is dense in $H$ and $W_L$ has open interior in $H$), hence we find $F_x \subset \Gamma_G = \Lambda^\infty$ such that
\[
\tau(x)W_L \subset W_L \tau(F_x) \quad \text{and} \quad W_L\tau(x) \subset \tau(F_x) W_L.
\]
Consequently,
\[
x\Sigma \subset \tau^{-1}(\tau(x\Sigma)) = \tau^{-1}(\Gamma_H \cap \tau(x)W_L) \subset \tau^{-1}(\Gamma_H \cap W_L \tau(F_x)) = \Sigma F_x, 
\]
and similarly $\Sigma x \subset F_x \Sigma$. This shows that $F \subset {\rm qComm}_{\Lambda^\infty}(\Sigma)$ and finishes the proof.
\end{proof}
Let us say that a lcsc group $G$ can be \emph{coupled} to a lcsc group $H$ if there exists a uniform lattice $\Gamma < G \times H$ which projects injectively to $G$ and densely to $H$. Then we have the following consequence of Theorem \ref{ThmMeyerGeneral}:
\begin{corollary} Let $G$ be a lcsc group which cannot be coupled to any non-compact connected Lie group. Then every Meyer set in $G$ is contained in a finite union of lattices.\qed
\end{corollary}

\section{Measurable vs. continuous quasi-cocycles}\label{AppendixCocycle}
Let $G$ be a lcsc group, $E$ be a Banach space and $\pi: G \to O(E)$ be a Borel homomorphism. Recall that a locally bounded Borel map $b: G \to E$ is called a \emph{Borel quasi-cocycle} if 
\begin{equation}
\label{def_Db}
D(b) := \sup_{g_1, g_2 \in G} \|b(g_1g_2)-b(g_1) - \pi(g_1)b(g_2)\| < \infty,
\end{equation}
and a \emph{continuous quasi-cocycle} if it is moreover continuous with respect to the norm topology on $E$. Burger and Monod \cite{BM1} define a group $G$ to have Property (TT) if every continuous quasi-cocycle on an $L^2$-space is bounded. On the contrary, Ozawa \cite{Oz11} defines Property (TT) by making the a priori stronger demand that every Borel quasi-cocycle on an $L^2$-space be bounded.

In our definition of Property (FF$L^p$) we follow Ozawa's convention to work with Borel quasi-cocycles. However, let us point out that the two definitions are actually equivalent in view of the following result which we record here for lack of an explicit reference:
\begin{theorem}\label{TTvsTT} If $b: G \to E$ is a Borel quasi-cocycle with respect to $\pi: G \to O(E)$, then there exists a continuous quasi-cocycle $b': G \to E$ which is at uniformly bounded distance from $b$. In particular, $G$ has Property (TT) in the sense of Burger and Monod iff it has Property (TT) in the sense of Ozawa.
\end{theorem}
\begin{remark} The full strength of the assumptions above is actually not needed. We only need $\pi$ to be a uniformly bounded representations of $G$ over $E$, which is Borel and hence actually strongly continuous. Concerning $b$ we only need that it is weakly measurable, locally bounded and satisfies \eqref{def_Db}.
\end{remark}
\begin{proof}[Proof of Theorem \ref{TTvsTT}] We fix a left-Haar measure $m_G$ on $G$ and $\rho \in C_c(G)$ which satisfies $\rho \geq 0$ and $\int_G\rho\, dm_G = 1$. We then define a function $b_\rho: G \to E$ by
\[
b_\rho(g) := \int_G \rho(h) b(gh)  \, dm_G(h) = \int_G \rho(g^{-1}h) b(h) \, dm_G(h),
\]
where the integral is understood in the Gelfand--Pettis sense. We also set $K_\rho := {\rm supp}(\rho)$.

Since $\pi$ is uniformly bounded and $b$ is locally bounded the distance
\begin{eqnarray*}
\|b(g) - b_\rho(g)\| &=& \|b(g) - \int_G  \rho(h)b(gh)\, dm_G(h) \|\\
&=& \|\int_G \rho(h) (b(g)-b(gh) + \pi(g)b(h)) dm_G(h) - \pi(g) \int_G \rho(h) b(h)\, dm_H(h)\|\\
&\leq&D(b) + \|\pi(g)\|_{\rm op} \cdot \|\int_{K_\rho}  \rho(h) b(h) dm_H(h)\|
\end{eqnarray*}
is bounded uniformly in $g$, hence it remains to see that $b_\rho$ is a continuous quasi-cocycle. As for the quasi-cocycle property, we observe that for all $g, h \in G$ we have
\begin{eqnarray*}
\|b_\rho(gh) - b_\rho(g) - \pi(g)b_\rho(h)\| &\leq& \|b_\rho(gh) - b(g) - \pi(g)b_\rho(h)\| + \|b(g) - b_\rho(g)\| \\
&\leq& \int_G \|b(ghk) - b(g) - \pi(g)b(hk)\| \, \rho(k) dm_G(k) +  \|b(g) - b_\rho(g)\|.
\end{eqnarray*}
Now the first summand is bounded by $D(b)$ in view of \eqref{def_Db}, and the second term is uniformly bounded in view of the previous argument. To see continuity of $b_\rho$ we fix $\varepsilon>0$. Since $\rho$ is uniformly continuous on compacta we can find a symmetric identity neighbourhood $V$ in $g$ such that $|\rho(x)-\rho(s^{-1}x)| < \varepsilon$ for all $s \in V$. Now for $s \in V$ and $g \in G$ we have
\[
\|b_\rho(gs) - b_\rho(g)\| = \int_{G} (\rho(s^{-1}g^{-1}h) - \rho(g^{-1}h))b(h) dm_G(h).
\]
The first factor is bounded in absolute value by $\varepsilon$. Moreover, the integral vanishes unless $\{s^{-1}g^{-1}h, g^{-1}h\} \cap K_\rho \neq \emptyset$, which implies $h \in K_g' := g(K_\rho \cup VK_\rho)$. It follows that
\[
\|b_\rho(gs) - b_\rho(g)\| \leq  \sup\{\|b(h)\| \mid h \in K_g'\} \cdot \varepsilon,
\]
which yields the desired continuity.
\end{proof}

\newcommand{\etalchar}[1]{$^{#1}$}

\end{document}